\documentclass[a4paper, 10pt, twoside, notitlepage]{amsart}

\usepackage{amsmath,amscd}
\usepackage{amssymb}
\usepackage{amsthm}
\usepackage{comment}
\usepackage{graphicx, xcolor}

\usepackage{mathrsfs}
\usepackage[linktocpage,ocgcolorlinks, linkcolor=blue]{hyperref}

\usepackage{bm}
\usepackage{bbm}
\usepackage{url}

\usepackage[utf8]{inputenc}
\usepackage{mathtools,amssymb}
\usepackage{esint}
\usepackage{tikz}
\usepackage{dsfont}
\usepackage{relsize}
\usepackage{url}
\usepackage{xcolor}
\usepackage{graphicx}
\usepackage{mathrsfs}
\usepackage[shortlabels]{enumitem}
\usepackage{lineno}
\usepackage{amsmath}
\usepackage{enumitem}
\usepackage{amsthm} 
\usepackage{verbatim}
\usepackage{dsfont}
\numberwithin{equation}{section}

\allowdisplaybreaks

\mathtoolsset{showonlyrefs}

\graphicspath{{images/}}

\newtheorem{theorem}{Theorem}[section]
\newtheorem{lemma}[theorem]{Lemma}
\newtheorem{definition}[theorem]{Definition}
\newtheorem{corollary}[theorem]{Corollary}

\newtheorem{proposition}[theorem]{Proposition}

\newtheorem{question}{Question}
\newtheorem{remark}[theorem]{Remark}
\newtheorem{assumption}{Assumption}

\newtheorem*{centertext}{}

\title[The Calder\'on problem for nonlocal wave equations]{Optimal Runge approximation for nonlocal wave equations and unique determination of polyhomogeneous nonlinearities}

\author[Y.-H.~Lin]{Yi-Hsuan Lin}
\address{Department of Applied Mathematics, National Yang Ming Chiao Tung University, Hsinchu, Taiwan \& Fakult\"at f\"ur Mathematik, University of Duisburg-Essen, Essen, Germany}
\email{yihsuanlin3@gmail.com}

\author[T.~Tyni]{Teemu Tyni}
\address{Research Unit of Applied and Computational Mathematics, University of Oulu, Finland}
\email{teemu.tyni@oulu.fi}

\author[P. Zimmermann]{Philipp Zimmermann}
\address{Departament de Matem\`atiques i Inform\`atica, Universitat de Barcelona, Barcelona, Spain}
\email{philipp.zimmermann@ub.edu}

\keywords{Fractional Laplacian, wave equations, nonlinear PDEs, inverse problems, Runge approximation, very weak solutions, polyhomogeneous.}
\subjclass[2020]{Primary 35R30; secondary 26A33, 42B37.}

\newcommand{\R}{{\mathbb R}}

\newcommand{\N}{{\mathbb N}}

\newcommand{\eps}{\varepsilon}

\newcommand {\p} {\partial}

\newcommand{\LC}{\left(}
\newcommand{\RC}{\right)}
\newcommand{\wt}{\widetilde}

\newcommand{\schwartz}{\mathscr{S}}

\newcommand{\tempered}{\mathscr{S}^{\prime}}

\newcommand{\fourier}{\mathcal{F}}
\newcommand{\ifourier}{\mathcal{F}^{-1}}

\newcommand{\distr}{\mathscr{D}^{\prime}}
\newcommand{\norm}[1]{\lVert #1 \rVert}
\newcommand{\abs}[1]{\left\lvert #1 \right\rvert}

\newcommand{\weak}{\rightharpoonup}
\newcommand{\weakstar}{\overset{\ast}{\rightharpoonup}}

\begin{document}

	\maketitle
	\begin{abstract}
		The main purpose of this article is to establish the Runge-type approximation in $L^2(0,T;\widetilde{H}^s(\Omega))$ for solutions of linear nonlocal wave equations. To achieve this, we extend the theory of very weak solutions for classical wave equations to our nonlocal framework. This strengthened Runge approximation property allows us to extend the existing uniqueness results for Calder\'on problems of linear and nonlinear nonlocal wave equations in our earlier works. Furthermore, we prove unique determination results for the Calder\'on problem of nonlocal wave equations with polyhomogeneous nonlinearities.
	\end{abstract}

	\tableofcontents

	\section{Introduction}
	\label{sec: introduction}
	
	In recent years, inverse problems for nonlocal partial differential equations (PDEs) have been extensively studied in the literature. The first work in this field is \cite{GSU20}, in which the authors considered the so-called \emph{Calder\'on problem} for the \emph{fractional Schr\"odinger equation}
	\begin{equation}
		\label{eq: fractional schroedinger}
		((-\Delta)^s+q)u=0\text{ in }\Omega,
	\end{equation}
	where $\Omega\subset \R^n$ is a bounded domain. Here $(-\Delta)^s$ denotes the \emph{fractional Laplacian} for $0<s<1$, and $q\in L^{\infty}(\Omega)$ is a bounded potential. In this problem, one asks whether it is possible to uniquely recover the potential $q$ from the \emph{Dirichlet-to-Neumann (DN) map} 
	\begin{equation}
		\label{eq: DN map schroedinger}
		\Lambda_q f\vcentcolon = (-\Delta)^s u_f\big|_{\Omega_e},
	\end{equation}
	where $\Omega_e=\R^n\setminus\overline{\Omega}$ denotes the exterior of $\Omega$, $f\colon\Omega_e\to\R$ is a given Dirichlet data, and $u_f\colon\R^n\to\R$ is the unique solution of \eqref{eq: fractional schroedinger} with $\left. u_f\right|_{\Omega_e}=f$. In \cite{GSU20}, the authors found that the fractional Laplacian satisfies the \emph{unique continuation principle (UCP)} that asserts:
	
	\begin{centertext}[\textbf{UCP}]
		Let $s\in\R_+\setminus\N$, $t\in\R$ and $V\subset\R^n$ be an open set. If $u\in H^t(\R^n)$ satisfies 
		\[
		u=(-\Delta)^s u=0\text{ in }V \text{ then } u\equiv 0 \text{ in } \R^n.
		\] 
	\end{centertext}
	In \cite{GSU20} the UCP has been shown for the range $0<s<1$ and via iteration by the Laplacian the UCP extends to the range $s\in \R_+\setminus\N$.
	Furthermore, they observe that the fractional Laplacian has, as a consequence of a duality argument (Hahn--Banach theorem) and the UCP, the so-called \emph{Runge approximation property}. For any open set $W\subset \Omega_e$, this property can be phrased in two alternative ways:
	\begin{enumerate}[(i)]
		\item\label{L2 approx} The \emph{Runge set}
		\[
		\mathcal{R}_W\vcentcolon =\big\{ u_{f}\big|_{\Omega} \,;\, f\in C^\infty_c (W) \big\}
		\]
		is dense in $L^2(\Omega)$, where $u_f$ is the unique solution to \eqref{eq: DN map schroedinger} with exterior value $f\in C_c^{\infty}(W)$ (cf.~\cite{GSU20}).
		\item\label{Hs approx} The \emph{Runge set}
		\[
		\mathscr{R}_W\vcentcolon =\big\{ u_{f}-f \,;\, f\in C^\infty_c (W) \big\}
		\]
		is dense in $\widetilde{H}^s(\Omega)$ (cf.~\cite{RS17}).
	\end{enumerate}
	Together with a suitable integral identity, the result \ref{L2 approx} allowed the authors of \cite{GSU20} to uniquely recover bounded potentials, whereas the property \ref{Hs approx} was  used in \cite{RS17} to recover certain singular potentials.
	
	The above strategy to establish uniqueness for Calder\'on-type inverse problems of elliptic or parabolic nonlocal equations has lately been investigated in several research articles, such as \cite{GLX,cekic2020calderon,CLL2017simultaneously,LL2020inverse,LL2022inverse, LZ2023unique,KLZ-2022,KLW2022,LRZ2022calder,Semilinear-nonlocal-wave-paper,LLU2023calder,CGRU2023reduction,LLU2023calder,RZ-unbounded,RZ2022LowReg,CRTZ-2022,LZ2024uniqueness,feizmohammadi2021fractional,feizmohammadi2021fractional_closed,FKU24}. Some articles of this list consider the detection of linear perturbations as in the above fractional Schr\"odinger equation \eqref{eq: fractional schroedinger}, while others allowed nonlinear perturbations, or even studied the identification of leading order coefficients in the main nonlocal term in the considered PDE. We also point out a recent monograph \cite{LL_nonlocal_IP} for a comprehensive introduction to this research field.

	\subsection{The mathematical model and main results}
	
	In this article, we study Calder\'on-type inverse problems for linear and nonlinear \emph{nonlocal wave equations (NWEQs)} formulated generically as
	\begin{equation}
		\label{eq: main}
		\begin{cases}
			\partial_t^2u +(-\Delta)^su+f(x,u)=0 &\text{ in }\Omega_T,\\
			u=\varphi &\text{ in }(\Omega_e)_T,\\
			u(0)= \partial_t u(0)=0  &\text{ in }\R^n,
		\end{cases}
	\end{equation}
	where $f\colon \Omega\times \R\to\R$ is a suitable nonlinearity. Here we use the notation
	\begin{equation}
		A_t \vcentcolon = A\times (0,t), \text{ for any }A\subset\R^n\text{ and }t>0.
	\end{equation} 
	Let us note that nonlocal wave equations such as \eqref{eq: main} arise in a special case of \emph{peridynamics}  --- the theory of studying dynamics of materials with discontinuities such as fractures (see \cite{peridynamics}). Meanwhile, for the well-posedness of \eqref{eq: main}, we need the compatibility condition to match the initial value such that the exterior data satisfies $\varphi(x,0)=0$ in $\Omega_e$.
	
	By \cite[Theorem~3.1 and 3.6]{Semilinear-nonlocal-wave-paper} the equation \eqref{eq: main} is well-posed for regular exterior values $\varphi\colon\Omega_e\to\R$, when $f(x,\tau)=q(x)\tau$ with $q\in L^p(\Omega)$, where $1\leq p\leq\infty$ satisfies
	\begin{equation}
		\label{eq: restrictions on p}
		\begin{cases}
			\frac{n}{s}\leq p\leq \infty, &\, \text{if }\, 2s< n,\\
			2<p\leq \infty,  &\, \text{if }\, 2s= n,\\
			2\leq p\leq \infty, &\, \text{if }\, 2s\geq n,
		\end{cases}
	\end{equation}
	or $f$ satisfies the following assumption:
	\begin{assumption}\label{main assumptions on nonlinearities}
		We say that a Carath\'eodory function $f\colon \Omega\times \R\to\R$ is a \emph{weak nonlinearity} if it satisfies the following conditions:
		\begin{enumerate}[(i)]
			\item\label{prop f} $f$ has partial derivative $\partial_{\tau}f$, which is a Carath\'eodory function, and there exists $a\in L^p(\Omega)$ such that
			\begin{equation}
				\label{eq: bound on derivative}
				\left|\partial_\tau f(x,\tau)\right|\lesssim a(x)+|\tau|^r
			\end{equation}
			for all $\tau\in\R$ and a.e. $x\in\Omega$. Here the exponents $p$ and $r$ satisfy the restrictions \eqref{eq: restrictions on p}
			and
			\begin{equation}
				\label{eq: cond on r}
				\begin{cases}
					0\leq r<\infty, &\, \text{if }\, 2s\geq n,\\
					0\leq r\leq \frac{2s}{n-2s}, &\, \text{if }\, 2s< n,
				\end{cases}
			\end{equation}
			respectively. Moreover, $f$ fulfills the integrability condition $f(\cdot,0)\in L^2(\Omega)$.
			\item\label{prop F} There is a constant $C_1>0$ such that the function $F\colon \Omega\times\R\to\R$ defined via
			\[
			F(x,\tau)=\int_0^\tau f(x,\rho)\,d\rho
			\]
			satisfies $F(x,\tau)\geq -C_1$ for all $\tau\in\R$ and $x\in\Omega$.
		\end{enumerate}
	\end{assumption}
	Observe that given a function $0\leq q\in L^\infty(\Omega)$, an example of a nonlinearity $f$, which satisfies the conditions in Assumption~\ref{main assumptions on nonlinearities}, is a fractional power type nonlinearity $f(x,\tau)=q(x)|\tau|^r\tau$ for $r\geq 0$ with $r$ satisfying \eqref{eq: cond on r}. We refer readers to \cite[Section 3]{Semilinear-nonlocal-wave-paper} for more details.
	
	Assuming the well-posedness of \eqref{eq: main} for suitable nonlinearities $f$, as a generalization of the Calder\'on problem for the fractional Schr\"odinger equation, we aim to determine the nonlinearity $f(x,\tau)$ from the DN map $\Lambda_f$ related to \eqref{eq: main}, which can be formally defined by 
	\begin{equation}
		\Lambda_{f}\varphi\vcentcolon  = (-\Delta)^s u_\varphi  \big|_{(\Omega_e)_T},
	\end{equation}
	where $u_\varphi\colon\R^n_T\to\R$ denotes the unique solution to \eqref{eq: main} (cf.~eq.~\eqref{eq: DN map schroedinger}). Next, let us make a few remarks on this inverse problem.
	\begin{enumerate}[(a)]
		\item\label{linear case intro} \textit{Linear perturbations:} In \cite{KLW2022}, the uniqueness of this inverse problem has been established in the case $f(x,u)=q(x)u$ with $q\in L^{\infty}(\Omega)$. In \cite[Corollary~1.4]{Semilinear-nonlocal-wave-paper}, it has been shown that uniqueness still holds if $0\leq q\in L^p(\Omega)$ with $p$ satisfying \eqref{eq: restrictions on p}.
		\item\label{seminlinear case intro} \textit{Semilinear perturbations:} In \cite[Theorem~1.1]{Semilinear-nonlocal-wave-paper}, we showed that uniqueness holds even when
		\begin{enumerate}
			\item[(A)]\label{assump 1 intro} $f$ is a weak nonlinearity (see~Assumption~\ref{main assumptions on nonlinearities}) with $F,r$ satisfying 
			\begin{enumerate}
				\item[(A1)]\label{assump 2 intro} $F\geq 0$,
				\item[(A2)]\label{assump 3 intro} $0<r\leq 1$
			\end{enumerate}
			\item[(B)]\label{assump 4 intro} and $f$ is $(r+1)-$homogeneous.
		\end{enumerate}
	\end{enumerate}
	All of these results rely on a $L^2(\Omega_T)$ Runge approximation property of linear nonlocal wave equations (cf.~\ref{L2 approx}):
	\begin{proposition}[{Runge approximation, \cite[Proposition~4.1]{Semilinear-nonlocal-wave-paper}}]
		\label{prop: runge old paper}
		Let $\Omega\subset \R^n$ be a bounded Lipschitz domain, $W\subset\Omega_e$ an arbitrary open set, $s>0$ a non-integer and $T>0$. Suppose that $q\in L^p(\Omega)$ is nonnegative\footnote{This assumption is included for simplicity and the result remains true, for example, if one assumes instead $q\in L^{\infty}(\Omega)$.}, where $p$ is given by \eqref{eq: restrictions on p}.
		Consider the \emph{Runge set} 
		\begin{align*}
			\mathcal{R}_W\vcentcolon =\big\{ u_{\varphi}\big|_{\Omega_T} \,;\, \varphi\in C^\infty_c (W_T) \big\},
		\end{align*}
		where $u_\varphi\in C([0,T];H^s(\R^n))\cap C^1([0,T]; L^2(\R^n))$ is the unique solution to 
		\begin{equation}\label{eq:runge with zero initial}
			\begin{cases}
				\partial_t^2u +(-\Delta)^su+qu = 0 &\text{ in }\Omega_T,\\
				u=\varphi  &\text{ in }(\Omega_e)_T,\\
				u(0)= \partial_t u(0)=0 &\text{ in }\Omega.
			\end{cases}
		\end{equation}
		Then $\mathcal{R}_W$ is dense in $L^2(\Omega_T)$.
	\end{proposition}
	
	The main goal of this article is to answer the question raised in \cite[p.~3]{zimmermann2024calderon}:
	\begin{question}
		Let us adopt all the notation of Proposition~\ref{prop: runge old paper}. Is the Runge set 
		\begin{equation}
			\mathscr{R}_W\vcentcolon =\left\{ u_{\varphi}-\varphi \,;\, \varphi\in C^\infty_c (W_T) \right\}
		\end{equation}
		dense in $L^2(0,T;\widetilde{H}^s(\Omega))$?
	\end{question}
	An affirmative answer to this question will be established by extending the theory of very weak solutions to linear nonlocal wave equations (Section~\ref{sec: very weak sol.}) and it forms the core contribution of our article as it provides an essential tool for studying the subsequent inverse problems. It reads as follows:
	
	\begin{theorem}[Optimal Runge approximation]\label{thm: runge}
		Let $\Omega\subset \R^n$ be a bounded Lipschitz domain, $W\subset\Omega_e$ an arbitrary open set, $s>0$ a non-integer, $T>0$ and $q\in L^p(\Omega)$ with $p$ satisfying the restrictions \eqref{eq: restrictions on p}. Consider the \emph{Runge set} 
		\begin{align*}
			\mathscr{R}_W\vcentcolon =\left\{ u_\varphi-\varphi \,;\, \varphi\in C^\infty_c (W_T) \right\}\subset L^2(0,T;\widetilde{H}^s(\Omega)),
		\end{align*}
		where $u_\varphi\in C([0,T];H^s(\R^n))\cap C^1([0,T]; L^2(\R^n))$ is the unique solution to \eqref{eq:runge with zero initial}. Then $\mathscr{R}_W$ is dense in $L^2(0,T;\widetilde{H}^s(\Omega))$.
	\end{theorem}
    We call this result optimal Runge approximation as the approximation space coincides with the natural space of solutions in our problem, similarly to the elliptic case, see~\ref{Hs approx} on page 2.
    
	With this Runge approximation and a suitable integral identity, similar to the one demonstrated in \cite{KLW2022} or \cite{zimmermann2024calderon}, we can recover $L^p$-regular linear perturbations, which are not necessarily nonnegative (cf.~\ref{linear case intro}).
	
	\begin{theorem}\label{thm: uniqueness potential}
		Let $\Omega\subset \R^n$ be a bounded Lipschitz domain, $W_1,W_2\subset\Omega_e$ be nonempty open sets, $s>0$ a non-integer, $T>0$ and $q_j\in L^p(\Omega)$ with $p$ satisfying the restrictions \eqref{eq: restrictions on p} for $j=1,2$. Let $\Lambda_{q_j}$ be the DN maps of
		\begin{equation}
			\begin{cases}
				\partial_t^2u +(-\Delta)^su + q_ju=0 &\text{ in }\Omega_T,\\
				u_j=\varphi  &\text{ in }(\Omega_e)_T,\\
				u_j(0)=\p_t u_j(0)=0 &\text{ in }\Omega,
			\end{cases}
		\end{equation}
		for $j=1,2$, satisfying
		\begin{equation}\label{same DN map in thm linear}
			\Lambda_{q_1}\varphi\big|_{(W_2)_T}=\Lambda_{q_2}\varphi\big|_{(W_2)_T},   \text{ for any }\varphi \in C^\infty_c((W_1)_T).
		\end{equation}
		Then there holds $q_1=q_2$ in $\Omega$.
	\end{theorem}
	
	Next, we present our main results on inverse problems for nonlinear nonlocal wave equations. To this purpose, let us introduce two different types of nonlinearities having a polyhomogeneous structure.
	
	\begin{definition}[Polyhomogeneous nonlinearities]
		\label{def: polyhomogeneous} 
		Assume that $\mathfrak{r}\vcentcolon =\LC r_k\RC_{k=1}^\infty\subset \R_+$ is a strictly monotonically increasing sequence, $\Omega\subset\R^n$ an open set and $f\colon\Omega\times \R\to\R$ a Carath\'eodory function.
		\begin{enumerate}[(i)]
			\item\label{serial} We call $f$ \emph{serially $\mathfrak{r}$-polyhomogeneous}, if we have
			\begin{equation}
            \label{eq: pointwise conv ser poly}
			f(x,\tau) = \sum_{k\geq 1} f_k(x,\tau)
			\end{equation}
			for a.e. $x\in\Omega$ and all $\tau\in\R$, where each expansion coefficient $f_k\colon \Omega\times \R\to\R$ is a  Carath\'eodory function that is $(r_k+1)$-homogeneous in the $\tau$-variable and there exists $b_k\geq 0$ such that
			\begin{equation}
				\label{eq: continuity nemytskii}
				\left|f_k(x,\tau)\right|\leq b_k|\tau|^{r_k+1}
			\end{equation}
			for a.e. $x\in\Omega$ and all $\tau\in\R$.
			\item\label{asymp} We call $f$ \emph{asymptotically $\mathfrak{r}$-polyhomogeneous}, denoted as
			\[
			f(x,\tau) \sim \sum_{k\geq 1} f_k(x,\tau),
			\]
			if each expansion coefficient $f_k\colon \Omega\times \R\to\R$ is a  Carath\'eodory function
			satisfying \eqref{eq: continuity nemytskii}, is $(r_k+1)$-homogeneous in the $\tau$-variable and for all $N\in\N_{\geq 2}$ there is a constant $C_N>0$ such that
			\begin{equation}
				\label{eq: asymp poly}
				\bigg| f(x,\tau) - \sum_{k=1}^{N-1} f_k(x,\tau)\bigg| \leq C_N |\tau|^{r_N+1}
			\end{equation}
            for all $|\tau|\leq 1$
            and $|f(x,\tau)|\leq C_\infty |\tau|^{r_\infty+1}$ for all $|\tau|>1$ for a.e. $x\in\Omega$.
		\end{enumerate}
	\end{definition}
	
	With Assumption \ref{main assumptions on nonlinearities} and Definition \ref{def: polyhomogeneous} at hand, we can state our main results for inverse problems of nonlinear nonlocal wave equations.

	\begin{theorem}[Recovery of expansion coefficients]
		\label{Thm: recovery of nonlinearity}
		Let $\Omega \subset \R^n$ be a bounded Lipschitz domain, $T>0$, and $s>0$ a non-integer. Let $W_1,W_2\subset \Omega_e$ be nonempty open sets. Suppose the nonlinearities $f^{(j)}$ satisfy the conditions in Assumption \ref{main assumptions on nonlinearities} with $F^{(1)},F^{(2)}\geq 0$, $a^{(1)},a^{(2)}\in L^{\infty}(\Omega)$, $r^{(1)}=r^{(2)}=r_{\infty}$, and are serially or asymptotically $\mathfrak{r}$-polyhomogeneous such that the corresponding (strictly) monotonically increasing sequence $\mathfrak{r}\vcentcolon =(r_k)_{k\in\N}$  fulfills $\mathfrak{r}\subset (0,r_\infty]$ (see~Definition~\ref{def: polyhomogeneous}).
		Assume that the DN maps $\Lambda_{f^{(j)}}$ of
		\begin{equation}
			\label{eq: nonlinear wave equation j=1,2}
			\begin{cases}
				\partial_t^2u +(-\Delta)^su + f^{(j)}(x,u)=0 &\text{ in }\Omega_T,\\
				u=\varphi  &\text{ in }(\Omega_e)_T,\\
				u(0)=\p_t u(0)=0 &\text{ in }\Omega,
			\end{cases}
		\end{equation}
		for $j=1,2$, coincide;
		\begin{equation}\label{same DN map in thm}
			\Lambda_{f^{(1)}}\varphi\big|_{(W_2)_T}=\Lambda_{f^{(2)}}\varphi\big|_{(W_2)_T}  \text{ for any }\varphi \in C^\infty_c((W_1)_T). 
		\end{equation}
		\begin{enumerate}[(i)]
			\item\label{main thm ser poly} If $f^{(j)}$ are serially $\mathfrak{r}$-polyhomogeneous with 
			\begin{equation}
				\label{eq: growth for bk}
				L_j\vcentcolon =\limsup_{k\to\infty}\frac{b_{k+1}^{(j)}}{b_k^{(j)}}<1,
			\end{equation}
			then we have $f^{(1)}=f^{(2)}$ a.e.~$x\in\Omega$ and all $\tau\in\R$.
			\item\label{main thm asymp poly} If $f^{(j)}$ are asymptotically $\mathfrak{r}$-polyhomogeneous.
            Then $f_k^{(1)}(x,\tau)=f_k^{(2)}(x,\tau)$ for a.e.~$x\in\Omega$, all $\tau\in\R$ and all $k\in\N$.
		\end{enumerate}
	\end{theorem}
	
    \begin{remark}
    The earlier article \cite{Semilinear-nonlocal-wave-paper} by the authors required nonnegativity of the potential $q\in L^p(\Omega)$ or the stronger assumption that $q\in L^\infty(\Omega)$. This was mainly used for the construction of solutions $u_\eps$, $\eps>0$, to a related parabolically regularized problem, cf. \cite[p.15, Proof of Claim 4.2]{Semilinear-nonlocal-wave-paper}, and for deducing their convergence properties as $\eps\to 0$. On the inverse problem for the nonlinear equation, the present manuscript generalizes results of \cite{Semilinear-nonlocal-wave-paper} from homogeneous to polyhomogeneous nonlinearities, allowing the recovery of the expansion coefficients term-by-term.
    
    \end{remark}

	\subsection{Comparison to inverse problems for local wave equations}
	
	The inverse problem of recovering coefficient functions in linear wave equations is classical, and the first results in this direction were by Belishev and Kurylev using a boundary control method~\cite{Be87,BeKu92}, see also  \cite{KKL01}.
	Nowadays, there are also results in reconstructing Riemannian manifolds for linear wave and other equations from partial data boundary measurements, such as in \cite{AKKLT,Helin,Isozaki,KKLO,KrKL,KOP,LO}.
	However, in the linear case, these results are based on the boundary control method, which requires the (lower order) coefficient functions to be time-independent or real-analytic in time, see \cite{Esk}.
    The recent results in \cite{AFO2022,AFO2024} have relaxed the assumptions on the smoothness of the coefficients in the linear local problem under certain curvature bounds, and in particular for metrics close to the Minkowski metric.
	Thanks to the unique continuation and Runge approximation properties in the nonlocal case, such restrictions are not needed, and one can consider much lower regularity coefficients, such as in Theorem~\ref{Thm: recovery of nonlinearity} (which also includes the linear case, see also \cite{Semilinear-nonlocal-wave-paper}).
	Moreover, nonlocal wave equations often enjoy infinite speed of propagation of the solution, and hence one can recover coefficients in larger domains than in the local case, where causality forces restrictions on the possible domains of reconstruction. 
    There are few results on low-regularity nonlinearities in the local setting, such as \cite{Johansson}, where the local elliptic version $-\Delta u + a(x,u)=0$ is considered, and it is proved that a nonlinearity $a(x,\tau)$ which is $L^r(\Omega)$, $r>n/2$, in $x$ and $C^{1,\alpha}(\R)$ in $\tau$ can be uniquely determined.
	
	Recently, inverse problems for nonlinear (local) wave equations have become mainstream. Let us elaborate on several works in this research field. The nonlinear (self-)interaction of waves will generate new waves, and this effect can be treated as a benefit in solving related inverse problems in the hyperbolic and elliptic settings. The seminal work \cite{KLU18} demonstrated that local measurements can be utilized to recover global topology and differentiable structure uniquely for a semilinear wave equation with a quadratic nonlinearity.  Furthermore, general semilinear wave equations on Lorentzian manifolds and related inverse problems were studied for the Einstein-Maxwell equation in  \cite{LUW18} and \cite{LUW17arXiv}, respectively. We also refer readers to \cite{dHUW18,KLOU14,LLL2024determining,LLPT20,LLPT21,LLPT24,WZ19} for different studies, including simultaneous recovery, geometrical inverse problems, and references therein.
	
	\subsection{Organization of the paper}
	
	The remainder of this paper is arranged as follows.
	In Section \ref{sec: pre}, we introduce several function spaces used in this paper. In Section \ref{sec: very weak sol.}, we show that there exists a unique very weak solution to linear nonlocal wave equations. In Section \ref{sec: runge+inverse}, we prove a stronger version of Runge approximation and use it to solve an inverse problem for linear nonlocal wave equations. Finally, we prove Theorem \ref{Thm: recovery of nonlinearity} in Section \ref{sec: nonlinear wave}.

	\section{Preliminaries}\label{sec: pre}
	
	In this section, we introduce some notation, which will be used throughout the whole article, and use this occasion to recall several basic facts on fractional Sobolev spaces as well as the fractional Laplacian.
	
	Throughout this article, we use the following convention for the Fourier transform
	\begin{equation}
		\fourier u(\xi) \vcentcolon =\hat{u}(\xi)\vcentcolon = \int_{\R^n} u(x)e^{-\mathrm{i}x \cdot \xi} \,dx
	\end{equation}
	for functions $u\colon\R^n\to\R$ for example in the Schwartz space $\schwartz(\R^n)$, where $\mathrm{i}$ is the imaginary unit. By duality, the Fourier transform can be extended to the space of tempered distributions $\tempered(\R^n) =(\schwartz(\R^n))^*$, and we use the same notation for it. We denote the inverse Fourier transform by $\ifourier$. 
	
	For any $s\in\R$, we let $H^{s}(\R^{n})$ stand for the \emph{fractional Sobolev space}, which consists of all tempered distributions $u$ such that
	\begin{equation}\notag
		\|u\|_{H^{s}(\mathbb{R}^{n})}\vcentcolon = \left\|\langle D\rangle^s u \right\|_{L^2(\R^n)}<\infty,
	\end{equation}
	where $\langle D\rangle^s$ denotes the Bessel potential of order $s$ with the Fourier symbol $\langle\xi\rangle^s\vcentcolon = \LC 1+|\xi|^2\RC^{s/2}$. We shall also need the following local versions of the fractional Sobolev spaces $H^s(\R^n)$:
	\begin{align*}
		H^{s}(\Omega) & \vcentcolon =\left\{ u|_{\Omega} \ ; \, u\in H^{s}(\mathbb{R}^{n})\right\},\\
		\widetilde{H}^{s}(\Omega) & \vcentcolon =\text{closure of \ensuremath{C_{c}^{\infty}(\Omega)} in \ensuremath{H^{s}(\mathbb{R}^{n})}},
	\end{align*}
	where $\Omega\subset\R^n$ is an open set and
	$H^{s}(\Omega)$ naturally endowed with the quotient norm
	\[
	\|u\|_{H^{s}(\Omega)}\vcentcolon =\inf\left\{ \|U\|_{H^{s}(\mathbb{R}^{n})} \ ; \,  U\in H^{s}(\mathbb{R}^{n})\mbox{ and }U|_{\Omega}=u\right\}.
	\]
	Furthermore, we set
	\[
	\widetilde{L}^2(\Omega)\vcentcolon = \widetilde{H}^0(\Omega)\quad \text{and}\quad  \|\cdot\|_{\widetilde{L}^2(\Omega)}\vcentcolon =\|\cdot\|_{\widetilde{H}^0(\Omega)}=\|\cdot\|_{L^2(\R^n)}.
	\]
	Let us also emphasize that if $\Omega$ is a Lipschitz domain, then one has the following identification
	\begin{equation}
		(\widetilde{H}^s(\Omega))^*=H^{-s}(\Omega).
	\end{equation}
	
	The \emph{fractional Laplacian} of order $s> 0$ is the homogeneous counterpart of the Bessel potential and hence is the Fourier multiplier
	\begin{equation}\label{eq:fracLapFourDef}
		(-\Delta)^{s} u = \ifourier\big( \abs{\xi}^{2s}\widehat{u}(\xi)\big), \text{ for }u\in \schwartz(\R^n).
	\end{equation}
	It is not difficult to see that an equivalent norm on $H^s(\R^n)$ is given by
	\begin{equation}
		\label{eq: equivalent norm on Hs}
		\|u\|_{H^s(\R^n)}^*= \|u\|_{L^{2}(\mathbb{R}^{n})}+\big\|(-\Delta)^{s/2}u  \big\|_{L^{2}(\mathbb{R}^{n})},
	\end{equation}
	and the fractional Laplacian is a bounded linear operator as a map $(-\Delta)^{s}\colon H^{t}(\R^n) \to H^{t-2s}(\R^n)$ for all $s \geq0$ and $t \in \R$. In fact, one can also write $(-\Delta)^s = (-\Delta)^k (-\Delta)^{\alpha}$, where $s=k+\alpha$ with $k=\lfloor s \rfloor \in \N \cup \{0\}$ and $\alpha=s-k\in (0,1)$.

	Using the Hardy--Littlewood--Sobolev lemma and H\"older's inequality, one can easily see that the following Poincaré inequality holds:

	\begin{proposition}[{Poincar\'e inequality (cf.~\cite[Lemma~5.4]{RZ-unbounded})}]\label{prop:Poincare ineq} Let $\Omega\subset\R^n$ be a bounded domain. For any $s> 0$, there exists $C>0$ such that
		\begin{equation}
			\label{eq: poincare ineq}
			\|u\|_{L^2(\Omega)}\leq C  \big\|(-\Delta)^{s/2}u  \big\|_{L^2(\R^n)}, \text{ for any }u\in \widetilde{H}^s(\Omega).
		\end{equation}
	\end{proposition}

	Taking into account that \eqref{eq: equivalent norm on Hs} is an equivalent norm on $\widetilde{H}^s(\Omega)$, the Poincaré inequality (Propositions~\ref{prop:Poincare ineq}) ensures the following simple lemma, which will be used throughout the whole article.
	\begin{lemma}
		\label{lemma: equivalent norm on tilde spaces}
		Let $\Omega\subset\R^n$ be a bounded domain and $s\geq 0$. Then an equivalent norm on $\widetilde{H}^s(\Omega)$ is given by
		\begin{equation}
			\label{eq: equivalent norm on tilde spaces}
			\|u\|_{\widetilde{H}^s(\Omega)}=\big\|(-\Delta)^{s/2}u \big\|_{L^2(\R^n)},
		\end{equation}
		which is induced by the inner product
		\begin{equation}
			\langle u,v\rangle_{\widetilde{H}^s(\Omega)}=\big\langle (-\Delta)^{s/2}u,(-\Delta)^{s/2}v \big\rangle_{L^2(\R^n)}.
		\end{equation}
	\end{lemma}
	
	Finally, let us mention that if $X$ is a Banach space, then we denote by $C^k([a,b]\,;X)$ and $L^p(a,b\,;X)$ ($k\in\N,1\leq p\leq \infty$) the space of $k$-times continuously differentiable functions and the space of measurable functions $u\colon (a,b)\to X$ such that $t\mapsto \|u(t)\|_X\in L^p([a,b])$, respectively. These Banach spaces are endowed with the norms
	\begin{equation}
		\label{eq: Bochner spaces}
		\begin{split}
			\|u\|_{L^p(a,b\,;X)}&\vcentcolon = \bigg(\int_a^b\|u(t)\|_{X}^p\,dt\bigg)^{1/p}<\infty,\\
			\|u\|_{C^k([a,b];X)}&\vcentcolon = \sup_{0\leq \ell\leq k}\left\|\partial_t^{\ell}u\right\|_{L^{\infty}(a,b;X)}
		\end{split}
	\end{equation} 
	(with the usual modification for $p=\infty$).

	\section{Existence and uniqueness of very weak solutions to linear nonlocal wave equations}\label{sec: very weak sol.}
	
	The purpose of this section is to extend the well-established theory of \emph{very weak solutions} to linear wave equations in our nonlocal setting. In Section~\ref{subsec: def very weak sol}, we motivate and present the rigorous definition of these solutions. Afterward, in Section~\ref{sec: spectral lemma} we formulate a spectral theoretic lemma, which we need later in Section~\ref{subsec: construction very weak without potential} for the construction of very weak solutions to NWEQs without a potential term. In Section~\ref{subsec: Very weak solutions with potential}, we then establish via a fixed point argument the well-posedness theory of very weak solutions to linear NWEQs with a nonzero potential. Finally, in Section~\ref{subsec: properties of very weak sol} we discuss some properties of very weak solutions. In particular, we show that all weak solutions are very weak solutions, which in turn are distributional solutions.

	Throughout the whole section, $\Omega\subset\R^n$ denotes a fixed bounded Lipschitz domain and $s\in \R_+\setminus \N$. As usual $\langle\cdot,\cdot\rangle$ denotes the duality pairing between $\widetilde{H}^s(\Omega)$ and $H^{-s}(\Omega)$. These Hilbert spaces are endowed with the norms $\|\cdot\|_{\widetilde{H}^s(\Omega)}$, introduced in Lemma~\ref{lemma: equivalent norm on tilde spaces}, and
	\begin{equation}\label{negative norm}
		\|G\|_{H^{-s}(\Omega)} = \sup\big\{ |\langle G,v\rangle| ; \, v\in\widetilde{H}^s(\Omega), \, \|v\|_{\widetilde{H}^s(\Omega)}=1\big\}.
	\end{equation}

	\subsection{Definition of very weak solutions}
	\label{subsec: def very weak sol}
	
	Next, we introduce the notion of very weak solutions to linear NWEQs with possibly a nonzero potential $q$.
	\begin{definition}
		\label{def: very weak sol}
		Let $F\in L^2(0,T;H^{-s}(\Omega))$, $u_0\in \wt{L}^2(\Omega)$, $u_1\in H^{-s}(\Omega)$ and $q\in L^{p}(\Omega)$, where the exponent $p$ satisfies the restriction \eqref{eq: restrictions on p}. A function  $u\colon \R^n_T\to \R$ is called a \emph{very weak solution} of 
		\begin{equation}
			\label{eq: very weak sol}
			\begin{cases}
				\partial_t^2u +(-\Delta)^su + qu=F &\text{ in }\Omega_T,\\
				u=0  &\text{ in }(\Omega_e)_T,\\
				u(0)=u_0, \quad \partial_t u(0)=u_1 &\text{ in }\Omega,
			\end{cases}
		\end{equation}
		if $u\in C([0,T];\wt{L}^2(\Omega))\cap C^1([0,T];H^{-s}(\Omega))$ satisfies
		\begin{equation}
			\label{eq: def of very weak sol}
			\begin{split}
				\int_0^T \left\langle u(t),G(t)\right\rangle_{L^2(\Omega)}\,dt =\int_0^T\left\langle F(t),v(t)\right\rangle_{L^2(\Omega)}\, dt +\left\langle u_1,v(0)\right\rangle-\left\langle u_0,\partial_t v(0)\right\rangle_{L^2(\Omega)},
			\end{split}
		\end{equation}
		for all $G\in L^2(0,T;\widetilde{L}^2(\Omega))$, where $v\in C([0,T];\widetilde{H}^s(\Omega))\cap C^1([0,T];\wt{L}^2(\Omega))$ is the unique (weak) solution of the backward equation 
		\begin{equation}
			\label{eq: weak sol}
			\begin{cases}
				\partial_t^2v +(-\Delta)^sv + qv=G &\text{ in }\Omega_T,\\
				v=0  &\text{ in }(\Omega_e)_T,\\
				v(T)=\partial_t v(T)=0 &\text{ in }\Omega
			\end{cases}
		\end{equation}
		(see~\cite[Theorem~3.1]{Semilinear-nonlocal-wave-paper}).
	\end{definition}

	\begin{remark}
		We recall that if $F\in L^2(0,T;\widetilde{L}^2(\Omega))$, $u_0\in \widetilde{H}^s(\Omega)$ and $u_1\in \wt{L}^2(\Omega)$, then a \emph{weak solution} of \eqref{eq: very weak sol} is a function $u\in C([0,T];\widetilde{H}^s(\Omega))\cap C^1([0,T];\wt{L}^2(\Omega))$ such that
		\begin{equation}
			\label{eq: weak solution}
			\frac{d}{dt}\left\langle\partial_t u,w \right\rangle_{L^2(\Omega)}+ \big\langle (-\Delta)^{s/2}u,(-\Delta)^{s/2}w \big\rangle_{L^2(\R^n)}+\left\langle qu,w \right\rangle_{L^2(\Omega)}=\langle F,w\rangle_{L^2(\Omega)},
		\end{equation}
		for all $w\in \widetilde{H}^s(\Omega)$ in the sense of $\distr((0,T))$. Often, we refer to weak solutions or very weak solutions simply as solutions, because the source term in the relevant PDEs determines which notion of solutions we invoke. Moreover, weak solutions are always very weak solutions, as we will see later in Proposition~\ref{prop: weak sol are very weak sol}.
	\end{remark}
	
	\begin{remark}
		Let us emphasize that the restriction on the exponent $p$ comes from the observation that if $q\in L^p(\Omega)$ and $u\in C([0,T];L^2(\Omega))$, then we have $qu\in L^2(0,T;H^{-s}(\Omega))$ (see \eqref{eq: continuity estimate potential} in the proof of Theorem~\ref{eq: thm well-posedness with q nonzero}).
	\end{remark}
	
	Before proceeding, let us give a formal motivation for imposing the identity  \eqref{eq: def of very weak sol}. We test the equation $\partial_t^2u +(-\Delta)^su + qu=F$ in $H^{-s}(\Omega)$ by the solution $v\in C([0,T];\widetilde{H}^s(\Omega))$ to \eqref{eq: weak sol} and integrate the resulting identity from $t=0$ to $t=T$. This gives
	\begin{equation}
		\label{eq: derivation of very weak formula}
		\begin{split}
			&\quad \, \int_0^T  \left\langle \partial_t^2 u(t),v(t)\right\rangle \,dt\\
			&=-\int_0^T \left\langle (-\Delta)^s u(t),v(t)\right\rangle\,dt -\int_0^T\left\langle qu(t),v(t)\right\rangle\,dt +\int_0^T \langle F(t),v(t)\rangle \,dt\\
			&=-\int_0^T\left\langle (-\Delta)^s v(t),u(t)\right\rangle\,dt -\int_0^T\left\langle qv(t),u(t)\right\rangle\,dt +\int_0^T \langle F(t),v(t)\rangle \,dt\\
			&= \int_0^T \left\langle \partial_t^2 v(t),u(t)\right\rangle \,dt - \int_0^T \langle G(t),u(t)\rangle\,dt +\int_0^T \langle F(t),v(t)\rangle \,dt.
		\end{split}
	\end{equation}
	Using an integration by parts, the term on the left-hand side and the first term on the right-hand side can be rewritten as
	\[
	\begin{split}
		\int_0^T \left\langle \partial_t^2 u(t),v(t)\right\rangle dt&=-\int_0^T \left\langle \partial_t u(t),\partial_t v(t)\right\rangle dt +\left\langle \partial_t u(T),v(T)\right\rangle -\left\langle \partial_t u(0),v(0)\right\rangle\\
		&=-\int_0^T \left\langle \partial_t u(t),\partial_t v(t)\right\rangle \,dt-\left\langle u_1,v(0)\right\rangle,\\
		\int_0^T \left\langle \partial_t^2 v(t),u(t)\right\rangle dt&=-\int_0^T \left\langle \partial_t v(t),\partial_t u(t)\right\rangle dt +\left\langle \partial_t v(T),u(T)\right\rangle -\left\langle \partial_t v(0),u(0)\right\rangle\\
		&=-\int_0^T \left\langle \partial_t u(t),\partial_t v(t)\right\rangle \,dt-\left\langle \partial_t v(0),u_0\right\rangle.
	\end{split}
	\]
	Inserting these identities into \eqref{eq: derivation of very weak formula}, we get \eqref{eq: def of very weak sol}. 
	
	\begin{remark}
		Note that the above computations are only formal because the integration by parts identities require better regularity than $\partial_t u\in L^2(0,T;H^{-s}(\Omega))$ 
        for the integral $\int_0^T \left\langle \partial_t u(t),\partial_t v(t)\right\rangle \,dt$ to make sense.
	\end{remark}
 
	\subsection{A spectral theoretic lemma}
	\label{sec: spectral lemma}
	
	For the construction of very weak solutions to linear nonlocal wave equations, we will need the following elementary spectral theoretic result. Even though the argument is standard, we offer the proof in Appendix~\ref{sec: proof of spectral lemma} for readers' convenience.
	\begin{lemma}
		\label{spectral lemma}
		Let $\Omega\subset\R^n$ be a bounded Lipschitz domain and $s\in \R_+ \setminus \N$. There exists a sequence of (Dirichlet) eigenvalues of the fractional Laplacian $(-\Delta)^s$ satisfying $0<\lambda_1\leq \lambda_2\leq \ldots$ with $\lambda_k\to\infty$ as $k\to\infty$ such that the corresponding eigenfunctions $\LC \phi_k\RC_{k\in\N}\subset\widetilde{H}^s(\Omega)$ have the following properties:
		\begin{enumerate}[(i)]
			\item\label{spectral prop 1} $\LC\phi_k\RC_{k\in\N}$ is an orthonormal basis of $\wt{L}^2(\Omega)$,
			\item\label{spectral prop 2} $\big(\lambda_k^{-1/2}\phi_k\big)_{k\in\N}$ is an orthonormal basis of $\widetilde{H}^s(\Omega)$,
			\item\label{spectral prop 3} $\big( \lambda_k^{1/2}\phi_k\big)_{k\in\N}$ is an orthonormal basis of $H^{-s}(\Omega)$.
		\end{enumerate}
	\end{lemma}
	
	\subsection{Very weak solutions to linear nonlocal wave equations without potential}
	\label{subsec: construction very weak without potential}
	
	The main purpose of this section is to prove the following well-posedness result.
	
	\begin{theorem}[Well-posedness of NWEQ with $q=0$]
		\label{eq: thm well-posedness q=0}
		Let $F\in L^2(0,T;H^{-s}(\Omega))$, $u_0\in \wt{L}^2(\Omega)$ and $u_1\in H^{-s}(\Omega)$. Then there exists a unique solution of
		\begin{equation}
			\label{eq: very weak sol without potential}
			\begin{cases}
				\partial_t^2u +(-\Delta)^su =F &\text{ in }\Omega_T,\\
				u=0  &\text{ in }(\Omega_e)_T,\\
				u(0)=u_0, \quad \partial_t u(0)=u_1 &\text{ in }\Omega.
			\end{cases}
		\end{equation}
		Moreover, the following continuity estimate holds
		\begin{equation}
			\label{eq: continuity estimates}
			\begin{split}
				\|u(t)\|_{L^2(\Omega)}+\left\|\partial_t u(t)\right\|_{H^{-s}(\Omega)}\leq C\big(\left\|u_0\right\|_{L^2(\Omega)}+\left\|u_1\right\|_{H^{-s}(\Omega)}+\|F\|_{L^2(0,T;H^{-s}(\Omega))}\big),
			\end{split}
		\end{equation}
		for some $C>0$ and for all $0\leq t\leq T$.
	\end{theorem}
	
	\begin{proof}

		We use the Fourier method to show the existence of a solution to \eqref{eq: very weak sol without potential}, that is we make the ansatz
		\begin{equation}
			\label{eq: guessed solution}
			u(t)=\sum_{k=1}^{\infty}c_k(t)\phi_k
		\end{equation}
		and for later convenience we set
		\[
		u_m(t)\vcentcolon=\sum_{k=1}^m c_k(t)\phi_k,
		\]
		for any $m\in \N$.
		For $u$ to satisfy \eqref{eq: very weak sol without potential} in $H^{-s}(\Omega)$, the coefficient $c_k$, $k\in\N$, needs to solve the initial value problem
		\begin{equation}
			\label{eq: initial value problem for coeff}
			\begin{cases}
				c''_k(t)+\lambda_k c_k(t)=F_k(t), \\
				c_k(0)=u_0^k, \, c'_k(0)=u_1^k
			\end{cases}
		\end{equation}
		for $0<t<T$, where we set
		\[
		u_0^k=\left\langle u_0,\phi_k\right\rangle_{L^2(\Omega)},\, u_1^k=\left\langle u_1,\phi_k\right\rangle\text{ and }F_k(t)=\left\langle F(t),\phi_k\right\rangle.
		\]
		By Duhamel's principle, for any $k\in\N$, the coefficients $c_k$ are given by
		\begin{equation}
			\label{eq: coefficients}
			\begin{split}              c_k(t)&=u_0^k\cos\big(\lambda_k^{1/2}t\big)+\lambda_k^{-1/2}u_1^k \sin\big(\lambda_k^{1/2}t\big) \\
            &\qquad+\lambda_k^{-1/2}\int_0^tF_k(\tau)\sin\big(\lambda_k^{1/2}(t-\tau)\big)\,d\tau.
			\end{split}
		\end{equation}

		\textit{Step 1.} We first show that for any $t\in [0,T]$, the series in \eqref{eq: guessed solution} converges in $\wt{L}^2(\Omega)$. By \cite[Corollary~5.10]{Brezis}, we only need to ensure that $c_k(t)\in \ell^2$. By Jensen's inequality, we may estimate
		\begin{equation}
			\label{eq: ck for uniform estimate}
			\begin{split}
				|c_k(t)|^2&\leq 3\bigg(|u_0^k|^2+\lambda_k^{-1}|u_1^k|^2+\lambda_k^{-1}\left|\int_0^t F_k(\tau)\sin\big(\lambda_k^{1/2}(t-\tau)\big)\,d\tau\right|^2\bigg)\\
				&\leq 3\bigg(|u_0^k|^2+\lambda_k^{-1}|u_1^k|^2+t\lambda_k^{-1}\int_0^t |F_k(\tau)|^2\,d\tau\bigg),
			\end{split}
		\end{equation}
		for any $k\in\N$.  As $u_0\in \wt{L}^2(\Omega)$ and $(\phi_k)_{k\in\N}$ is an orthonormal basis in $\wt{L}^2(\Omega)$, \cite[Corollary~5.10]{Brezis} implies $(u_0^k)_{k\in\N}\in \ell^2$ with
		\begin{equation}
			\label{eq: u0}
			\|u_0\|_{L^2(\Omega)}^2=\sum_{k=1}^{\infty}|u_0^k|^2.
		\end{equation}
		Similarly, we know by \eqref{eq: characterization of dual norm} that 
		$(\lambda_k^{-1/2}u_1^k)_{k\in\N}\in \ell^2$ with
		\begin{equation}
			\label{eq: u1}
			\|u_1\|_{H^{-s}(\Omega)}^2=\sum_{k=1}^{\infty}\lambda_k^{-1}|u_1^k|^2.
		\end{equation}
		On the other hand, the formula \eqref{eq: characterization of dual norm} shows that if $G\in L^2(0,T;H^{-s}(\Omega))$, then $(\lambda_k^{-1}G_k)_{k\in\N}\in L^2(0,T;\ell^2)$. Additionally, by Tonelli's theorem, the integral and sum can be exchanged so that
		\begin{equation}
			\label{eq: L2 in time l2 space norm}
			\|G\|_{L^2(0,T;H^{-s}(\Omega))}^2=\|\lambda_k^{-1/2}G_k\|_{L^2(0,T;\ell^2)}^2=\sum_{k=1}^{\infty}\lambda_k^{-1}\int_0^T |G_k(t)|^2\,dt.
		\end{equation}
		This estimate can be applied to $F\in L^2(0,T;H^{-s}(\Omega))$. Hence, to sum up, we have
		\begin{equation}
			\label{eq: uniform estimate for coeff}
			\|c_k(t)\|^2_{\ell^2}\leq 3\big( \|u_0\|_{L^2(\Omega)}^2+\|u_1\|_{H^{-s}(\Omega)}^2+T\|F\|_{L^2(0,T;H^{-s}(\Omega))}^2\big).
		\end{equation}
		Now, again invoking \cite[Corollary~5.10]{Brezis} and Lemma~\ref{spectral lemma}, we deduce \eqref{eq: guessed solution} converges in $\wt{L}^2(\Omega)$ for any $t\in [0,T]$
		and 
		\begin{equation}
			\label{eq: L2 norm of sol at fixed time}
			\|u(t)\|_{L^2(\Omega)}^2=\|c_k(t)\|_{\ell^2}^2\leq 3\big(\|u_0\|_{L^2(\Omega)}^2+\|u_1\|_{H^{-s}(\Omega)}^2+T\|F\|_{L^2(0,T;H^{-s}(\Omega))}^2\big).
		\end{equation}
		
		\medskip
		
		\textit{Step 2.} We first show that $u\in C([0,T];\wt{L}^2(\Omega))$. To see this, it is enough to show that $u_m\in C([0,T];\wt{L}^2(\Omega))$ for $m\in\N$ and $u_m\to u$ in $\wt{L}^2(\Omega)$ as $m\to\infty$ uniformly in $0\leq t\leq T$. Note that we have
		\[
		\left\|u_m(t)-u_m(t')\right\|_{L^2(\Omega)}\leq \sum_{k=1}^m \left|c_k(t)-c_k(t')\right|,
		\]
		for $t,t'\in [0,T]$ and so $u_m\in C([0,T];\wt{L}^2(\Omega))$ as long as $c_k\in C([0,T])$. The first two terms of $c_k$ (see \eqref{eq: coefficients}) are continuous, hence it only remains to show that 
		\begin{equation}
			\label{eq: def of dk}
			d_k(t)\vcentcolon =\lambda_k^{-1/2}\int_0^t F_k(\tau)\sin\big(\lambda_k^{1/2}(t-\tau)\big)\,d\tau\in C([0,T]).
		\end{equation}
		Let us suppose that $t\geq t'$. Then we may calculate
		\begin{equation}
			\label{eq: continuity of dk}
			\begin{split}
				&\quad \, \left|d_k(t)-d_k(t')\right| \\
				&=\lambda_k^{-1/2}\left|\int_0^t F_k(\tau)\sin\big(\lambda_k^{1/2}(t-\tau)\big)\,d\tau-\int_0^{t'} F_k(\tau)\sin\big(\lambda_k^{1/2}(t'-\tau)\big)\,d\tau\right|\\
				&\leq \lambda_k^{-1/2}\int_{t'}^t |F_k(\tau)|\big|\sin\big(\lambda_k^{1/2}(t-\tau)\big)\big|\,d\tau\\
				&\quad \, +\lambda_k^{-1/2}\int_0^{t'}|F_k(\tau)|\big|\sin\big(\lambda_k^{1/2}(t'-\tau)\big)-\sin\big(\lambda_k^{1/2}(t-\tau)\big)\big|\,d\tau\\
				&\leq \lambda_k^{-1/2}\int_{t'}^t |F_k(\tau)|\,d\tau\\
				&\quad \, +\lambda_k^{-1/2}\int_0^{t'}|F_k(\tau)|\big|\sin\big(\lambda_k^{1/2}(t'-\tau)\big)-\sin\big(\lambda_k^{1/2}(t-\tau)\big)\big|\,d\tau\\
				&\leq |t-t'|^{1/2}\bigg(\int_{t'}^t \lambda_k^{-1}|F_k(\tau)|^2\,d\tau\bigg)^{1/2}\\
				&\quad \,+\bigg(\int_0^{t'}\lambda_k^{-1}|F_k(\tau)|^2\,d\tau\bigg)^{1/2}\bigg(\int_0^{t'}\big|\sin\big(\lambda_k^{1/2}(t'-\tau)\big)-\sin\big(\lambda_k^{1/2}(t-\tau)\big)\big|^2\,d\tau\bigg)^{1/2}\\
				&\leq |t-t'|^{1/2}\|F\|_{L^2(0,T;H^{-s}(\Omega))}\\
				&\quad\, +\|F\|_{L^2(0,T;H^{-s}(\Omega))}\bigg(\int_0^{t'}|\sin(\lambda_k^{1/2}(t'-\tau))-\sin(\lambda_k^{1/2}(t-\tau))|^2\,d\tau\bigg)^{1/2}.
			\end{split}
		\end{equation}
		In the first inequality, we wrote
		\[
		\sin\big(\lambda_k^{1/2}(t'-\tau)\big)=\big(\sin\big(\lambda_k^{1/2}(t'-\tau)\big)-\sin\big(\lambda_k^{1/2}(t-\tau)\big)\big)+\sin\big(\lambda_k^{1/2}(t'-\tau)\big)
		\]
		and in the fourth inequality used \eqref{eq: L2 in time l2 space norm}. Now, let $\epsilon>0$ and choose first $\rho >0$ such that $\rho^{1/2}\|F\|_{L^2(0,T;H^{-s}(\Omega))}<\epsilon/2$. Then choose $\delta>0$ such that
		\[
		\|F\|_{L^2(0,T;H^{-s}(\Omega))}T^{1/2}\delta<\epsilon/2.
		\]
		By uniform continuity of the sine function, we can find $\eta>0$ such that $|\sin x-\sin y|<\delta$, whenever $|x-y|<\eta$. Now, we set 
		\[
		\mu\vcentcolon=\min(\rho,\eta/\lambda_k^{1/2}).
		\]
		The above choices show that if $|t-t'|<\mu$ and $t\geq t'$, then 
		\[
		\begin{split}
			\left|d_k(t)-d_k(t')\right|&\leq \rho^{1/2}\|F\|_{L^2(0,T;H^{-s}(\Omega))}+\|F\|_{L^2(0,T;H^{-s}(\Omega))}\delta (t')^{1/2}\\
			&<\epsilon/2+\|F\|_{L^2(0,T;H^{-s}(\Omega))}\delta T^{1/2}\\
			&<\epsilon.
		\end{split}
		\]
		Interchanging the roles of $t$ and $t'$ shows that $d_k$ is uniformly continuous because the constant does not depend on the particular point $t$ or $t'$ (but the choice of $\mu$ depends on $k$).
        Hence, we have shown that $u_m\in C([0,T];\widetilde{L}^2(\Omega))$. 
		
		Next, we prove that $u_m $ converges uniformly to $u$ in $\wt{L}^2(\Omega)$ on $[0,T]$ as $m\to \infty$. Let $\ell\geq m$, then by \eqref{eq: ck for uniform estimate} we deduce that
		\[
		\begin{split}
			\left\|u_\ell(t)-u_m(t)\right\|_{L^2(\Omega)}^2&=\bigg\|\sum_{k=m+1}^\ell c_k(t)\phi_k\bigg\|_{L^2(\Omega)}^2 \\
			&=\sum_{k=m+1}^\ell \left|c_k(t)\right|^2\\
			&\leq C\sum_{k=m+1}^\ell \bigg(\left|u_0^k\right|^2+\lambda_k^{-1}\left|u_1^k\right|^2+T\lambda_k^{-1}\int_0^T \left|F_k(\tau)\right|^2\,d\tau\bigg),
		\end{split}
		\]
		for some constant $C>0$.
		Passing to the limit $\ell\to\infty$ gives
		\[
		\left\|u(t)-u_m(t)\right\|_{L^2(\Omega)}^2\leq  C\sum_{k=m+1}^{\infty} \bigg( \left|u_0^k \right|^2+\lambda_k^{-1}\left|u_1^k\right|^2+T\lambda_k^{-1}\int_0^T \left|F_k(\tau)\right|^2\,d\tau\bigg)
		\]
		for any $m\in\N$. The right-hand side is independent of $t\in [0,T]$ and by summability of the right-hand side (see \eqref{eq: u0}, \eqref{eq: u1} and \eqref{eq: L2 in time l2 space norm}) it needs to go to zero as $m$ tends to infinity. Thus, the convergence $u_m\to u$ in $\wt{L}^2(\Omega)$ as $m\to \infty$ is uniform in $t\in [0,T]$.\\
		
		\textit{Step 3.} Let us show that $u\in C^1([0,T];H^{-s}(\Omega))$. We first establish that $u_m\in C^1([0,T];H^{-s}(\Omega))$ for any $m\in\N$. Formally, by differentiating $c_k$, one may compute
		\begin{equation}
			\label{eq: chain rule}
			\begin{split}
				c'_k(t)&=-\lambda_k^{1/2}u_0^k\sin\big(\lambda_k^{1/2}t\big)+u_1^k\cos\big(\lambda_k^{1/2}t\big)+\int_0^t F_k(\tau)\cos\big(\lambda_k^{1/2}(t-\tau)\big)\,d\tau,
			\end{split}
		\end{equation}
		for any $k\in\N$. Taking derivatives for the first two terms does not cause any difficulty, but for the integral term, we need to justify it. Using the definition of $d_k$ from \eqref{eq: def of dk}, for any $t\in [0,T]$ and $h>0$ such that $t+h\in [0,T]$, one can compute
		\[
		\begin{split}
			&\quad \, \bigg|\frac{d_k(t+h)-d_k(t)}{h}-\int_0^tF_k(\tau)\cos\big(\lambda_k^{1/2}(t-\tau)\big)\,d\tau\bigg|\\
			&=\bigg|\int_t^{t+h}\lambda_k^{-1/2}F_k(\tau)\frac{\sin\big(\lambda_k^{1/2}(t+h-\tau)\big)}{h}\,d\tau\\
			&\quad \,\,   +\int_0^t\lambda_k^{-1/2}F_k(\tau) \\
			&\qquad \quad  \cdot  \bigg(\frac{\sin\big(\lambda_k^{1/2}(t+h-\tau)\big)-\sin\big(\lambda_k^{1/2}(t-\tau)\big)}{h}-\lambda_k^{1/2}\cos\big(\lambda_k^{1/2}(t-\tau)\big)\bigg)\,d\tau\bigg|\\
			&\leq \int_t^{t+h}\lambda_k^{-1/2}|F_k(\tau)|\bigg|\frac{\sin\big(\lambda_k^{1/2}(t+h-\tau)\big)}{h}\bigg|\,d\tau\\
			&\quad \,  +\int_0^t\lambda_k^{-1/2}\left|F_k(\tau)\right|\\
			&\qquad \quad \cdot \bigg|\frac{\sin\big(\lambda_k^{1/2}(t+h-\tau)\big)-\sin\big(\lambda_k^{1/2}(t-\tau)\big)}{h}-\lambda_k^{1/2}\cos\big(\lambda_k^{1/2}(t-\tau)\big)\bigg|\,d\tau\\
			& =\vcentcolon I_h+II_h.
		\end{split}
		\]
		Next, we show that both expressions $I_h$ and $II_h$ vanish as $h\to 0$.

		For $I_h$, by the triangle inequality, we have
		\[
		\begin{split}
			I_h&\leq \int_t^{t+h}\lambda_k^{-1/2}\left|F_k(\tau)\right|\bigg|\frac{\sin\big(\lambda_k^{1/2}(t+h-\tau)\big)-\sin\big(\lambda_k^{1/2}(t-\tau)\big)}{h}\bigg|\,d\tau\\
			&\quad \,  +\underbrace{\int_t^{t+h}\lambda_k^{-1/2}\left|F_k(\tau)\right|\bigg|\frac{\sin\big(\lambda_k^{1/2}(t-\tau)\big)}{h}\bigg| d\tau}_{=\frac{1}{h}\int_t^{t+h}\lambda_k^{-1/2}\left|F_k(\tau)\right||\sin(\lambda_k^{1/2}(t-\tau))| d\tau}\\
			&=\int_t^{t+h}\lambda_k^{-1/2}\left|F_k(\tau)\right|\bigg|\frac{\sin\big(\lambda_k^{1/2}(t+h-\tau)\big)-\sin\big(\lambda_k^{1/2}(t-\tau)\big)}{h}\bigg|\,d\tau+o(1)\\
			&=\int_t^{t+h}\left|F_k(\tau)\right|\big|\cos\big(\lambda_k^{1/2}(\eta-\tau)\big)\big|d\tau+o(1)\\
			&\leq \int_t^{t+h}\left|F_k(\tau)\right|\,d\tau+o(1)\\
			&=o(1)
		\end{split}
		\]
		as $h\to 0$. In the first equality we used that $F_k\in L^2((0,T))$ and Lebesgue's differentiation theorem implies 
		\[
		\frac{1}{h}\int_t^{t+h}\lambda_k^{-1/2}\left|F_k(\tau)\right|\big|\sin(\lambda_k^{1/2}(t-\tau))\big| d\tau \to 0 \text{ as }h\to 0.
		\] 
		In the second equality we applied the mean value theorem, where $\eta\in (t,t+h)$, and in the last equality the absolute continuity of the Lebesgue integral.

		On the other hand, the fact that $II_h\to 0$ as $h\to 0$ is a simple application of Lebesgue's dominated convergence theorem. The same argument works out for $h<0$ and hence the $d_k$ is differentiable with the derivative
		\begin{equation}
			\label{eq: derivative of integral term}
			d'_k(t)=\int_0^tF_k(\tau)\cos\big(\lambda_k^{1/2}(t-\tau)\big)\,d\tau.
		\end{equation}
		Hence, we have proved the formula \eqref{eq: chain rule}.

		It remains to show that $u'_m\in C([0,T];H^{-s}(\Omega))$, but by the same argument as above it is enough to establish $c'_k\in C([0,T])$. The first two terms in $c'_k$ are clearly continuous and hence we only need to show the continuity of $d'_k$ given by the formula \eqref{eq: derivative of integral term}. Let $t\geq t'$. By the same computation as in \eqref{eq: continuity of dk} up to replacing $\sin$ by $\cos$ and forgetting the prefactor $\lambda_k^{-1/2}$, we have
		\[
		\begin{split}
			\left|d'_k(t)-d'_k(t')\right|&\leq \int_{t'}^t |F_k(\tau)|\,d\tau \\
			&\quad \, +\int_0^{t'}\left|F_k(\tau)\right|\big|\cos\big(\lambda_k^{1/2}(t'-\tau)\big)-\cos\big(\lambda_k^{1/2}(t-\tau)\big)\big|d\tau.
		\end{split}
		\]
		Now, let $\epsilon>0$. As $F_k\in L^2((0,T))$, by the absolute continuity of the Lebesgue integral, we can find $\delta>0$ such that the first term is smaller than $\epsilon/2$, whenever $|t-t'|<\delta$. On the other hand, we can find a $\rho>0$ such that there holds
		\[
		|\cos x-\cos y|<\epsilon/2\|F_k\|_{L^1((0,T))}
		\]
		whenever $|x-y|<\rho$. Let 
		\[
		\mu=\min(\delta,\rho/\lambda_k^{1/2}).
		\]
		Hence, if $|t-t'|<\mu$, then we have
		\[
		\big|\lambda_k^{1/2}(t'-\tau)-\lambda_k^{1/2}(t-\tau)\big|=\lambda_k^{1/2}|t-t'|<\rho
		\]
		and hence
		\[
		\left|d'_k(t)-d'_k(t')\right|<\epsilon/2+ \frac{\epsilon}{2\|F_k\|_{L^1((0,T))}}\int_0^{t'}|F_k(\tau)|\,d\tau<\epsilon.
		\]
		The very same argument holds when $t\leq t'$, and as all parameters $\delta,\rho,\mu$ are independent of $t$ and $t'$, we have shown that $d'_k$ are uniformly continuous on $[0,T]$. Hence, we have $u_m\in C^1([0,T];H^{-s}(\Omega))$ for all $m\in\N$.
		
		Now, by Lemma~\ref{spectral lemma} and the same arguments as in \eqref{eq: ck for uniform estimate}, we get 
		\begin{equation}
			\label{eq: upper bound for derivative of partial sums}
			\begin{split}
				\left\|u'_m(t)\right\|^2_{H^{-s}(\Omega)}&=\bigg\|\sum_{k=1}^m c'_k(t)\phi_k\bigg\|_{H^{-s}(\Omega)}^2\\
				&=\bigg\|\sum_{k=1}^m \lambda_k^{-1/2}c'_k(t)(\lambda_k^{1/2}\phi_k)\bigg\|_{H^{-s}(\Omega)}^2\\
				&=\sum_{k=1}^m\lambda_k^{-1}|c'_k(t)|^2\\
				&\leq 3\sum_{k=1}^m\lambda_k^{-1}\bigg(\lambda_k|u_0^k|^2+|u_1^k|^2+\bigg(\int_0^t|F_k(\tau)|\,d\tau\bigg)^2\bigg)\\
				&\leq 3\sum_{k=1}^m\bigg(|u_0^k|^2+\lambda_k^{-1}|u_1^k|^2+T\lambda_k^{-1}\int_0^T|F_k(\tau)|^2\,d\tau\bigg).
			\end{split}
		\end{equation}
		Using \eqref{eq: u0}, \eqref{eq: u1} and \eqref{eq: L2 in time l2 space norm}, this implies
		\begin{equation}
			\label{eq: uniform upper bound for derivative of partial sums}
			\left\|u'_m(t)\right\|_{H^{-s}(\Omega)}^2\leq 3 \big(\|u_0\|_{L^2(\Omega)}^2+\|u_1\|_{H^{-s}(\Omega)}^2+T\|F\|_{L^2(0,T;H^{-s}(\Omega)}^2\big).
		\end{equation}
		Furthermore, we have
		\begin{equation}
			\label{eq: uniform conv derivative}
			\begin{split}
				\left\|u'_\ell(t)-u'_m(t)\right\|^2_{H^{-s}(\Omega)}&= \sum_{k=m+1}^\ell\lambda_k^{-1}|c'_k(t)|^2\leq \sum_{k=m+1}^{\infty}\lambda_k^{-1}|c'_k(t)|^2
			\end{split}
		\end{equation}
		for all $\ell \geq m$. Observing that the right-hand side goes to zero as $m\to\infty$, we see that $( u_m(t))_{m\in\N}\subset H^{-s}(\Omega)$ is a Cauchy sequence in $H^{-s}(\Omega)$ and therefore converges to some unique limit $w(t) \in H^{-s}(\Omega)$. Passing to the limit $\ell\to\infty$ in \eqref{eq: uniform conv derivative} and using the estimates from \eqref{eq: upper bound for derivative of partial sums}, we get
		\[
		\left\|w(t)-u'_m(t)\right\|_{H^{-s}(\Omega)}^2\leq 3\sum_{k=m+1}^{\infty}\bigg(|u_0^k|^2+\lambda_k^{-1}|u_1^k|^2+T\lambda_k^{-1}\int_0^T|F_k(\tau)|^2\,d\tau\bigg)
		\]
		for any $m\in\N$. The sum on the right-hand side is independent of $t$ and hence the convergence $u'_m\to w$ in $H^{-s}(\Omega)$ as $m\to\infty$ is uniform in $t\in[0,T]$. It is well-known that this implies $u\in C^1([0,T];H^{-s}(\Omega))$ with $u'=w$. Furthermore, by \eqref{eq: uniform upper bound for derivative of partial sums} there holds
		\begin{equation}
			\label{eq: estimate derivative}
			\|u'\|_{L^{\infty}(0,T;H^{-s}(\Omega))}\leq C \left(\|u_0\|_{L^2(\Omega)}+\|u_1\|_{H^{-s}(\Omega)}+\|F\|_{L^2(0,T;H^{-s}(\Omega)}\right),
		\end{equation}
		for some $C>0$. Notice that this estimate, together with  \eqref{eq: L2 norm of sol at fixed time} establishes \eqref{eq: continuity estimates}.\\
		
		\noindent\textit{Step 4.} In this step, we show that $u$ is in fact a solution of \eqref{eq: very weak sol without potential}. First, let us note that by formally applying the Leibniz rule and \eqref{eq: coefficients}, one has
		\begin{equation}
			\label{eq: second derivative}
			\begin{split}
				c''_k(t)&=-\lambda_ku_0^k \cos\big(\lambda_k^{1/2}t\big)-\lambda_k^{1/2}u_1^k\sin\big(\lambda_k^{1/2}t\big)+F_k(t)\\
				&\quad \,-\lambda_k^{1/2}\int_0^t F_k(\tau)\sin\big(\lambda_k^{1/2}(t-\tau)\big)\,d\tau\\
				&=-\lambda_k c_k(t)+F_k(t).
			\end{split}
		\end{equation}
		Thus, $c_k$ indeed solves  \eqref{eq: initial value problem for coeff}. To see that the first equality sign in formula \eqref{eq: second derivative} holds, it is enough to show that $d'_k$ is differentiable with derivative 
		\begin{equation}
			\label{eq: second derivative integral term}
			d''_k(t)=F_k(t)-\lambda_k^{1/2}\int_0^t F_k(\tau)\sin\big(\lambda_k^{1/2}(t-\tau)\big)\,d\tau.
		\end{equation}
		We can repeat the same computation as for the first derivative. This time we have
		\[
		\begin{split}
			&\quad \, \bigg|\frac{d'_k(t+h)-d'_k(t)}{h}-F_k(t)+\lambda_k^{1/2}\int_0^tF_k(\tau)\sin\big(\lambda_k^{1/2}(t-\tau)\big)\,d\tau\bigg|\\
			&\leq \int_t^{t+h}\bigg|\frac{F_k(\tau)\cos\big(\lambda_k^{1/2}(t+h-\tau)\big)-F_k(t)}{h}\bigg|\,d\tau\\
			&\quad \,  +\int_0^t|F_k(\tau)|\bigg|\frac{\cos\big(\lambda_k^{1/2}(t+h-\tau)\big)-\cos\big(\lambda_k^{1/2}(t-\tau)\big)}{h}-\lambda_k^{1/2}\sin\big(\lambda_k^{1/2}(t-\tau)\big)\bigg|\,d\tau\\
			& = \vcentcolon III_h+IV_h.
		\end{split}
		\]
		The second term $IV_h$ again goes to zero by Lebesgue's dominated convergence theorem. For the first term $III_h$, we proceed similarly to $I$ above. This gives
		\[
		\begin{split}
			III_h&\leq \int_t^{t+h}|F_k(\tau)|\bigg|\frac{\cos\big(\lambda_k^{1/2}(t+h-\tau)\big)-\cos\big(\lambda_k^{1/2}(t-\tau)\big)}{h}\bigg|\,d\tau\\
			&\quad \, +\int_t^{t+h}\bigg|\frac{F_k(\tau)\cos\big(\lambda_k^{1/2}(t-\tau)\big)-F_k(t)}{h}\bigg|\,d\tau.
		\end{split}
		\]
		The second term goes to zero as $h\to 0$ by Lebesgue's differentiation theorem, and for the first term, one can use the mean value theorem and the absolute continuity of the integral to find $III_h\to 0$ as $h\to 0$. This proves \eqref{eq: second derivative integral term} and hence \eqref{eq: second derivative}.
		From the differential equation for $c_k$ we get $c''_k\in L^2((0,T))$. Note that as $F$ is not (in general) continuous, we generally do not have $c_k\in C^2([0,T])$. But in fact $c'_k\in H^1((0,T))$ and the fundamental theorem of calculus imply $c_k \in C^{1,1/2}([0,T])$. 
		
		Now, we wish to show that $u$ is a solution in the sense of Definition~\ref{def: very weak sol}. To this end, let us assume that $G\in L^2(0,T;\widetilde{L}^2(\Omega))$ and $v\in C([0,T];\widetilde{H}^s(\Omega))\cap C^1([0,T];\wt{L}^2(\Omega))$ is the unique solution to 
		\begin{equation}
			\label{eq: test function very weak sol}
			\begin{cases}
				\partial_t^2v +(-\Delta)^sv =G &\text{ in }\Omega_T,\\
				v=0  &\text{ in }(\Omega_e)_T,\\
				v(T)= \partial_t v(T)=0 &\text{ in }\Omega.
			\end{cases}
		\end{equation}
		Note that from this equation we also have $\partial_t^2 v\in L^2(0,T;H^{-s}(\Omega))$. By Lemma~\ref{spectral lemma} we can write
		\begin{equation}
			\label{eq: expansion of v}
			v(t)=\sum_{k=1}^{\infty}\alpha_k(t)\phi_k,
		\end{equation}
		where $\alpha_k=\langle v(t),\phi_k\rangle_{L^2(\Omega)}$. Furthermore, note that by our choice of $\phi_k$ and the inner product on $H^{-s}(\Omega)$ (see \eqref{eq: inner product on dual space}), there holds
		\begin{equation}
			\label{eq: coefficients of v}
			\alpha_k=\lambda_k^{-1/2}\big\langle v(t),\lambda_k^{-1/2}\phi_k\big\rangle_{\widetilde{H}^s(\Omega)}=\lambda_k^{1/2}\big\langle v(t),\lambda_k^{1/2}\phi_k\big\rangle_{H^{-s}(\Omega)}.
		\end{equation}
		The last equality is shown in \eqref{eq: expression for Gk}.
		For later convenience, we let $v_m$ be defined via
		\[
		v_m(t)=\sum_{k=1}^m \alpha_k(t)\phi_k.
		\]
		Moreover, we know that:
		\begin{enumerate}[(a)]
			\item\label{prop 1 for v} For any $t\in[0,T]$ one has $v_m(t)\to v(t)$ in $\widetilde{H}^s(\Omega)$ as $m\to\infty$.
			\item\label{prop 2 for v} There holds $\alpha_k\in C^1([0,T])\cap H^2((0,T))$ for any $k\in\N$.
			\item\label{prop 3 for v} For any $k\in\N$ the functions $\alpha_k$ solve
			\begin{equation}
				\label{eq: ODE for alpha k}
				\begin{cases}
					\alpha''_k(t)+\lambda_k \alpha_k(t)=G_k(t)\\
					\alpha_k(T)=\alpha'_k(T)=0,
				\end{cases}
			\end{equation}
			for $0<t<T$, where $G_k(t)=\langle G(t),\phi_k\rangle_{L^2(\Omega)}$.
		\end{enumerate}
		
		In fact, \ref{prop 1 for v} follows from Lemma~\ref{spectral lemma} and the regularity of $v$. The regularity $\alpha_k\in C^1([0,T])$ in \ref{prop 2 for v} follows from $v\in C^1([0,T];\wt{L}^2(\Omega))$. The claim that $\alpha_k\in H^2((0,T))$ can be seen as follows. First of all, $v''\in L^2(0,T;H^{-s}(\Omega))$ implies 
		\[
		-\int_0^T v'(t)\eta'(t)\,dt=\int_0^Tv''(t)\eta(t)\,dt\text{ in } H^{-s}(\Omega)
		\]
		for all $\eta\in C_c^{\infty}((0,T))$. Acting on this identity by $w\mapsto \lambda_k^{1/2}\langle w,\lambda_k^{1/2}\phi_k\rangle_{H^{-s}(\Omega)}\in (H^{-s}(\Omega))^*$ gives
		\[
		-\int_0^T\lambda_k^{1/2}\langle v'(t),\lambda_k^{1/2}\phi_k\rangle_{H^{-s}(\Omega)}\eta'(t)\,dt=\int_0^T \lambda_k^{1/2}\langle v''(t),\lambda_k^{1/2}\phi_k\rangle_{H^{-s}(\Omega)}\eta(t)\,dt.
		\]
		Now, the first factor on the left-hand side is nothing else than $\alpha'_k$, and thus
		\[
		\alpha''_k(t)=\lambda_k^{1/2}\langle v''(t),\lambda_k^{1/2}\phi_k\rangle_{H^{-s}(\Omega)}\in L^2((0,T)).
		\]
		This shows that $\alpha_k\in H^2((0,T))$ for $k\in\N$. The endpoint conditions in \ref{eq: ODE for alpha k} follow from $v(T)=v'(T)=0$. From the previous calculation, \eqref{eq: inner product on dual space} and \eqref{eq: test function very weak sol}, we can compute
		\[
		\begin{split}
			\alpha''_k(t)&=\lambda_k^{1/2}\big\langle v''(t),\lambda_k^{1/2}\phi_k\big\rangle_{H^{-s}(\Omega)}\\
			&=\lambda_k^{1/2}\big\langle v''(t)-G(t),\lambda_k^{1/2}\phi_k\big\rangle_{H^{-s}(\Omega)}+\lambda_k^{1/2}\big\langle G(t),\lambda_k^{1/2}\phi_k\big\rangle_{H^{-s}(\Omega)}\\
			&=\big\langle S(v''(t)-G(t)),S(\lambda_k\phi_k)\big\rangle_{\widetilde{H}^s(\Omega)}+\left\langle S(G(t)),S(\lambda_k\phi_k)\right\rangle_{\widetilde{H}^s(\Omega)}\\
			&=-\big\langle v(t),\phi_k\big\rangle_{\widetilde{H}^s(\Omega)}+\left\langle S(G(t)),\phi_k\right\rangle_{\widetilde{H}^s(\Omega)}\\
			&=-\lambda_k\langle v(t),\phi_k\rangle_{L^2(\Omega)}+\langle G(t),\phi_k\rangle_{L^2(\Omega)}\\
			&=-\lambda_k\alpha_k(t)+G_k(t),
		\end{split}
		\]
		where $S\colon H^{-s}(\Omega)\to\widetilde{H}^s(\Omega)$ is the source-to-solution map for the Dirichlet problem of the fractional Laplacian $(-\Delta)^s$ (see Appendix~ \ref{sec: proof of spectral lemma} for more details). This verifies that $\alpha_k$ satisfies \ref{prop 3 for v} . By $c_k\in C^1([0,T])\cap H^2((0,T))$, \eqref{eq: initial value problem for coeff}, \ref{prop 2 for v} and \ref{prop 3 for v} we may calculate
		\begin{equation}
			\label{eq: first pI}
			\begin{split}
				\int_0^T c''_k \alpha_k\,dt&=-\int_0^T c'_k\alpha'_k\,dt+c'_k(T)\alpha_k(T)-c'_k(0)\alpha_k(0)\\
				&=-\int_0^T c'_k\alpha'_k\,dt-u_1^k\alpha_k(0)
			\end{split}
		\end{equation}
		and
		\begin{equation}
			\label{eq: second pI}
			\begin{split}
				\int_0^T c_k \alpha''_k\,dt&=-\int_0^T c'_k\alpha'_k\,dt+c_k(T)\alpha'_k(T)-c_k(0)\alpha'_k(0)\\
				&=-\int_0^T c'_k\alpha'_k\,dt-u_0^k\alpha'_k(0).
			\end{split}
		\end{equation}
		Inserting \eqref{eq: second pI} into \eqref{eq: first pI} yields
		\begin{equation}
			\int_0^T c''_k \alpha_k\,dt=\int_0^T c_k\alpha''_k\,dt+u_0^k\alpha'_k(0)-u_1^k\alpha_k(0).
		\end{equation}
		Hence by \eqref{eq: initial value problem for coeff} and \ref{prop 3 for v}, we get
		\begin{equation}
			\label{eq: eq for coefficients well-posedness}
			\int_0^T (-\lambda_kc_k+F_k) \alpha_k\,dt=\int_0^T c_k(-\lambda_k\alpha_k+G_k)\,dt+u_0^k\alpha'_k(0)-u_1^k\alpha_k(0),
		\end{equation}
		or equivalently
		\begin{equation}
			\label{eq: integration by parts formula for coeff}
			\int_0^T F_k \alpha_k\,dt=\int_0^T c_kG_k\,dt+u_0^k\alpha'_k(0)-u_1^k\alpha_k(0).
		\end{equation}
		Summing this identity from $k=1$ to $k=N$ gives
		\begin{equation}\label{eq: integration by parts formula}
			\begin{split}
				\int_0^T \big\langle F^{(N)},v_N \big\rangle dt&=\int_0^T \big\langle u_N,G^{(N)}\big\rangle_{L^2(\Omega)}\,dt \\
				&\quad \, +\big\langle u_0^{(N)},v'_N(0)\big\rangle_{L^2(\Omega)}-\big\langle u_1^{(N)},v_N(0)\big\rangle,
			\end{split}
		\end{equation}
		where we set
		\[
		F^{(N)}=\sum_{k=1}^N F_k\phi_k,\quad G^{(N)}=\sum_{k=1}^N G_k\phi_k,\quad u_j^{(N)}=\sum_{k=1}^N u_j^k \phi_k
		\]
		for $j=0,1$. That  \eqref{eq: integration by parts formula for coeff} and \eqref{eq: integration by parts formula} are equivalent can be seen by Lemma~\ref {spectral lemma}. Next, note that $G\in L^2(0,T;\widetilde{L}^2(\Omega))$ and Lebesgue's dominated convergence theorem together with Parseval's identity ensure 
		\[
		\begin{split}
			\big\|G^{(N)}(t)\big\|_{L^2(\Omega)}^2&=\sum_{k=1}^N \big|\langle G,\phi_k \rangle_{L^2(\Omega)}\big|^2\\
			&\leq \sum_{k=1}^{\infty}\big|\langle G,\phi_k \rangle_{L^2(\Omega)}\big|^2=\|G(t)\|_{L^2(\Omega)}^2,
		\end{split}
		\]
		  which gives
		\[
		G^{(N)}\to G\text{ in }L^2(0,T;\widetilde{L}^2(\Omega))
		\]
		as $N\to\infty$. To see that $F^{(N)}\to F$ in $L^2(0,T;H^{-s}(\Omega))$, let us first observe that
		\begin{equation}
			\label{eq: identity for FN}
			\begin{split}
				F_k\phi_k&=\left\langle F,\phi_k \right\rangle \phi_k\\
				&=\langle SF,\phi_k\rangle_{\widetilde{H}^s(\Omega)}\phi_k\\
				&=\langle SF,S(\lambda_k\phi_k)\rangle_{\widetilde{H}^s(\Omega)}\phi_k\\
				&=\left\langle F,\lambda_k\phi_k\right\rangle_{H^{-s}(\Omega)}\phi_k\\
				&=\big\langle F,\lambda_k^{1/2}\phi_k\big\rangle_{H^{-s}(\Omega)}\lambda_k^{1/2}\phi_k.
			\end{split}
		\end{equation}
		As $\big( \lambda_k^{1/2}\phi_k\big)_{k\in\N}$ is an orthonormal basis in $H^{-s}(\Omega)$, we deduce from \eqref{eq: identity for FN} that there holds
		\[
		F^{(N)}(t)\to F(t)\text{ in }H^{-s}(\Omega)
		\]
		as $N\to\infty$. Thus, using Lebesgue's dominated convergence theorem and Parseval's identity for $G$, we get
		\[
		F^{(N)}\to F\text{ in }L^2(0,T;H^{-s}(\Omega))
		\]
		as $N\to\infty$. So, we can finally pass to the limit in \eqref{eq: integration by parts formula} to obtain
		\[
		\int_0^T\langle F(t),v(t)\rangle\,dt=\int_0^T \langle u(t),G(t)\rangle_{L^2(\Omega)}\,dt+\left\langle u_0,v'(0)\right\rangle_{L^2(\Omega)}-\left\langle u_1,v(0)\right\rangle.
		\]
		This establishes that $u$ is a solution to \eqref{eq: very weak sol without potential}. \\
		
		\noindent\textit{Step 5.} In this final step, we show that the constructed solution $u$ is unique. Suppose that $\widetilde{u}$ is another solution, then $U=u-\widetilde{u}$ is a solution to 
		\begin{equation}
			\begin{cases}
				\partial_t^2U +(-\Delta)^sU =0 &\text{ in }\Omega_T,\\
				U=0  &\text{ in }(\Omega_e)_T,\\
				U(0)=\partial_t U(0)=0 &\text{ in }\Omega.
			\end{cases}
		\end{equation}
		By Definition~\ref{def: very weak sol} this means
		\[
		\int_0^T \langle U(t),G(t)\rangle_{L^2(\Omega)}\,dt=0
		\]
		for all $G\in L^2(\Omega_T)$, but this clearly implies $U=0$ and hence $u=\widetilde{u}$.
	\end{proof}
	
	\subsection{Very weak solutions to linear nonlocal wave equations with potential}
	\label{subsec: Very weak solutions with potential}
	
	The purpose of this section is to extend the well-posedness theory of equation \eqref{eq: very weak sol without potential} to linear NWEQs with a nonzero potential.
	
	\begin{theorem}[Well-posedness nonlocal wave equation with potential]
		\label{eq: thm well-posedness with q nonzero}
		Let $F\in L^2(0,T;H^{-s}(\Omega))$, $u_0\in \wt{L}^2(\Omega)$ and $u_1\in H^{-s}(\Omega)$. Furthermore, assume that $q\in L^p(\Omega)$ with $p$ satisfying the restriction \eqref{eq: restrictions on p}.
		Then the problem
		\begin{equation}
			\label{eq: very weak sol with potential}
			\begin{cases}
				\partial_t^2u +(-\Delta)^su+qu =F &\text{ in }\Omega_T,\\
				u=0  &\text{ in }(\Omega_e)_T,\\
				u(0)=u_0, \quad \partial_t u(0)=u_1 &\text{ in }\Omega
			\end{cases}
		\end{equation}
		has a unique very weak solution $u\in C([0,T];\wt{L}^2(\Omega))\cap C^1([0,T];H^{-s}(\Omega))$. 
	\end{theorem}
	
	\begin{proof}
		Let us first note that $qu\in H^{-s}(\Omega)$ for any $u\in L^2(\Omega)$ as 
		\begin{equation}
			\label{eq: computation for L2 estimate}
			\begin{split}
				\bigg|\int_{\Omega}quv\, dx \bigg|&\leq \|qv\|_{L^2(\Omega)}\|u\|_{L^2(\Omega)}\\
				&\leq \|q\|_{L^{n/s}(\Omega)}\|v\|_{L^{\frac{2n}{n-2s}}(\Omega)}\|u\|_{L^2(\Omega)}\\
				&\leq C\|q\|_{L^{n/s}(\Omega)}\|v\|_{\widetilde{H}^s(\Omega)}\|u\|_{L^2(\Omega)}\\
				&\leq C\|q\|_{L^{p}(\Omega)}\left\|v \right\|_{\wt{H}^s(\Omega)}\|u\|_{L^2(\Omega)}
			\end{split}
		\end{equation}
		for all $v\in \widetilde{H}^s(\Omega)$.  The case $p=\infty$ is clear. In the case $\frac{n}{s}\leq p<\infty$ with $2s< n$ we used H\"older's inequality with
		\[
		\frac{1}{2}=\frac{n-2s}{2n}+\frac{s}{n},
		\]
		$L^{r_2}(\Omega)\hookrightarrow L^{r_1}(\Omega)$ for $r_1\leq r_2$ as $\Omega\subset\R^n$ is bounded, and Sobolev's inequality. In the case $2s>n$ one can use the embedding $H^s(\R^n)\hookrightarrow L^{\infty}(\R^n)$ and the boundedness of $\Omega$ to see that the estimate \eqref{eq: computation for L2 estimate} holds. In the case $n=2s$ one can use the boundedness of the embedding $\widetilde{H}^s(\Omega)\hookrightarrow L^{\bar{p}}(\Omega)$ for all $2\leq \bar{p}<\infty$, H\"older's inequality and the boundedness of $\Omega$ to get the final estimate \eqref{eq: computation for L2 estimate}. The aforementioned embedding in the critical case follows by \cite{Ozawa} and the Poincar\'e inequality. The above clearly implies that for any $u\in C([0,T];\wt{L}^2(\Omega))$, we have $qu\in L^2(0,T;H^{-s}(\Omega))$ with
		\begin{equation}
			\label{eq: continuity estimate potential}
			\|qu\|_{L^2(0,T;H^{-s}(\Omega))}\leq C\|q\|_{L^p(\Omega)}\|u\|_{L^2(\Omega_T)}.
		\end{equation}
		Now, we wish to use a fixed point argument to construct the solution to \eqref{eq: very weak sol without potential}.

		Via Theorem~\ref{eq: thm well-posedness q=0}, we can define 
		\begin{equation}
			\begin{split}
				S\colon C([0,T];\wt{L}^2(\Omega))&\to C([0,T];\wt{L}^2(\Omega))\cap C^1([0,T];H^{-s}(\Omega)), \quad v\mapsto u,
			\end{split}
		\end{equation} 
		where $u$ is the solution of 
		\begin{equation}
			\begin{cases}
				\partial_t^2u +(-\Delta)^su =F-qv &\text{ in }\Omega_T,\\
				u=0  &\text{ in }(\Omega_e)_T,\\
				u(0)=u_0, \quad \partial_t u(0)=u_1 &\text{ in }\Omega.
			\end{cases}
		\end{equation}
		Assume that $v_1,v_2\in C([0,T];\wt{L}^2(\Omega))$. Since the function $u=S(v^1)-S(v^2)$ solves 
		\[
		\begin{cases}
			\partial_t^2u +(-\Delta)^su =-q(v^1-v^2) &\text{ in }\Omega_T,\\
			u=0  &\text{ in }(\Omega_e)_T,\\
			u(0)= \partial_t u(0)=0 &\text{ in }\Omega,
		\end{cases}
		\]
		the energy estimate \eqref{eq: continuity estimates} (applied for the case $T=t$) yields
		\[
		\begin{split}
			\|u(t)\|_{L^2(\Omega)}+\left\|\partial_t u(t)\right\|_{H^{-s}(\Omega)}&\leq C\left\|q(v^1-v^2)\right\|_{L^2(0,t;H^{-s}(\Omega))},
		\end{split}
		\]
		for a.e. $t\in[0,T]$. Hence, by \eqref{eq: continuity estimate potential} we obtain
		\begin{equation}
			\label{eq: very weak sol with q est}
			\|u(t)\|_{L^2(\Omega)}\leq C\|q\|_{L^p(\Omega)}\left\|v^1-v^2 \right\|_{L^2(\Omega_t)},
		\end{equation}
		for $t\in [0,T]$. Next, introduce for $\theta>0$ the equivalent norm
		\[
		\|w\|_{\theta}\vcentcolon = \sup_{0\leq t\leq T}e^{-\theta t}\|w(t)\|_{L^2(\Omega)}
		\]
		on $C([0,T];\wt{L}^2(\Omega))$. Then from equation \eqref{eq: very weak sol with q est} we deduce that
		\[
		\begin{split}
			\|u(t)\|_{L^2(\Omega)}&\leq C\|q\|_{L^p(\Omega)}\left\|v^1-v^2\right\|_{\theta}\bigg(\int_0^te^{2\theta \tau}\,d\tau\bigg)^{1/2}\\
			&= \frac{C}{(2\theta)^{1/2}}(e^{2\theta t}-1)^{1/2}\|q\|_{L^p(\Omega)}\left\|v^1-v^2\right\|_{\theta}\\
			&\leq \frac{C}{(2\theta)^{1/2}}e^{\theta t}\|q\|_{L^p(\Omega)}\left\|v^1-v^2 \right\|_{\theta}
		\end{split}
		\]
		for $t\in [0,T]$. Dividing by $e^{\theta t}$ and taking the supremum over $[0,T]$, this implies
		\[
		\begin{split}
			\|u\|_{\theta}&\leq \frac{C}{(2\theta)^{1/2}}\|q\|_{L^p(\Omega)}\left\|v^1-v^2\right\|_{\theta}.
		\end{split}
		\]
		Remembering that $u=S(v^1)-S(v^2)$ and choosing $\theta>0$ such that
		\[
		C_0\vcentcolon = \frac{C}{(2\theta)^{1/2}}\|q\|_{L^p(\Omega)}<1,
		\]
		we get
		\[
		\|S(v^1)-S(v^2)\|_{\theta}\leq C_0\|v^1-v^2\|_{\theta}
		\]
		and thus $S$ is a contraction on the complete metric space $(C([0,T];\wt{L}^2(\Omega)),\|\cdot\|_{\theta})$. Therefore, we can apply the Banach fixed point theorem to deduce that $S$ has a unique fixed point $u\in C([0,T];\wt{L}^2(\Omega))\cap C^1([0,T];H^{-s}(\Omega))$. Hence, we have shown the existence of a unique very weak solution, and we can conclude the proof. 
	\end{proof}
	
	\subsection{Properties of very weak solutions}
	\label{subsec: properties of very weak sol}
	
	In this section, we establish some relations between various definitions of solutions to linear NWEQs, which can be seen as a consistency test of the introduced notions. In Proposition~\ref{prop: very weak sol are distribution} and \ref{prop: weak sol are very weak sol}, we show that
	\[
	 \text{ weak solution }\Rightarrow\text{ very weak solution }\Rightarrow\text{ distributional solution}
	\]
	and in Corollary~\ref{cor: regular very weak sol are weak sol} that 
	\[
	\text{ regular very weak solution = weak solution}.
	\]
	
	\begin{proposition}[Distributional solutions]\label{prop: very weak sol are distribution}
		Let $F\in L^2(0,T;H^{-s}(\Omega))$, $u_0\in \wt{L}^2(\Omega)$, $u_1\in H^{-s}(\Omega)$ and $q\in L^p(\Omega)$ with $p$ satisfying the restrictions \eqref{eq: restrictions on p}. The unique very weak solution of \eqref{eq: very weak sol with potential} is a distributional solution, that is, there holds
		\begin{equation}
			\label{eq: distributional solution}
			\int_{\Omega_T}u\LC \partial_t^2\varphi+(-\Delta)^s\varphi+q\varphi\RC dt=\int_0^T\langle F,\varphi\rangle\,dt+\langle u_0,\partial_t\varphi(0)\rangle_{L^2(\Omega)}-\langle u_1,\varphi(0)\rangle,
		\end{equation}
		for all $\varphi\in C_c^{\infty}([0,T)\times \Omega)$.
	\end{proposition}
	\begin{proof}
		Let us note that it is enough to prove the result for $q=0$ as the general case follows by replacing $F$ with $F-qu$.
		
		We use the same notation as in the proof of Theorem~\ref{eq: thm well-posedness q=0}, but this time the $\alpha_k$ are the coefficients in the expansion of $\varphi$ in the orthonormal basis $(\phi_k)_{k\in\N}$, that is $\alpha_k=\langle \varphi,\phi_k\rangle_{L^2(\Omega)}$. We start with the identity 
		\[
		\begin{split}
			\int_0^T c''_k \alpha_k\,dt=\int_0^T c_k\alpha''_k\,dt+u_0^k\alpha'_k(0)-u_1^k\alpha_k(0)
		\end{split}
		\]
		(see~\eqref{eq: eq for coefficients well-posedness}). Now, using the relation \eqref{eq: initial value problem for coeff}, we deduce the equality
		\[
		\int_0^T \LC-\lambda_k c_k+F_k\RC \alpha_k\,dt=\int_0^T c_k\alpha''_k\,dt+u_0^k\alpha'_k(0)-u_1^k\alpha_k(0).
		\]
		This is equivalent to
		\begin{equation}
			\label{eq: fourier coeff integral}
			\int_0^T c_k \LC\alpha''_k+\lambda_k \alpha_k\RC dt=\int_0^T F_k\alpha_k\,dt+u_1^k\alpha_k(0)-u_0^k\alpha'_k(0).
		\end{equation}
		Next note that
		\[
		\begin{split}
			\langle (-\Delta)^s\varphi(t),\phi_k\rangle_{L^2(\Omega)}&=\langle (-\Delta)^s\varphi(t),\phi_k\rangle_{L^2(\R^n)}\\
			&=\langle (-\Delta)^{s/2} \varphi(t),(-\Delta)^{s/2}\phi_k\rangle_{L^2(\R^n)}\\
			&=\lambda_k\langle \varphi(t),\phi_k\rangle_{L^2(\Omega)}\\
			&=\lambda_k\alpha_k(t)
		\end{split}
		\]
		for all $0\leq t\leq T$. Using $\langle \phi_k,\phi_\ell \rangle_{L^2(\Omega)}=\delta_{k\ell}$, we may write
		\[
		\begin{split}
			\bigg\langle\sum_{k=1}^m c_k\phi_k,\sum_{\ell=1}^m\langle (-\Delta)^s \varphi,\phi_{\ell}\rangle_{L^2(\Omega)}\phi_{\ell}\bigg\rangle_{L^2(\Omega)}=\sum_{k=1}^m\lambda_kc_k\alpha_k.
		\end{split}
		\]
		Since $u,\chi_{\Omega}(-\Delta)^s\varphi\in L^2(0,T;\widetilde{L}^2(\Omega))$, where $\chi_\Omega$ is the characteristic function of $\Omega$, the time integral of the left hand side converges to $\langle u,(-\Delta)^s\varphi\rangle_{L^2(\Omega_T)}$. We also have
		\[
		\int_0^T\bigg\langle \sum_{k=1}^m c_k\phi_k,\sum_{\ell=1}^m\langle \partial_t^2 \varphi,\phi_{\ell}\rangle_{L^2(\Omega)}\phi_{\ell}\bigg\rangle_{L^2(\Omega)}\,dt=\int_0^T \sum_{k=1}^m c_k\alpha''_k\,dt\to \int_{\Omega_T}u\partial_t^2\varphi\,dxdt
		\]
		as $m\to\infty$. Thus, summing the identity \eqref{eq: fourier coeff integral} from $k=1$ to $m$ and passing to the limit $m\to\infty$ yields
		\[
		\int_{\Omega_T}u(\partial_t^2\varphi+(-\Delta)^s\varphi)\,dxdt=\int_0^T\langle F,\varphi\rangle\,dt+\langle u_0,\partial_t\varphi(0)\rangle_{L^2(\Omega)}-\langle u_1,\varphi(0)\rangle.
		\]
		For the convergence of the term involving $F$ we refer to Step 4 in the proof of Theorem~\ref{eq: thm well-posedness q=0}. Hence, we can conclude the proof.
	\end{proof}

	\begin{proposition}[Weak solutions are very weak solutions]
		\label{prop: weak sol are very weak sol}
		Suppose that $F\in L^2(0,T;\widetilde{L}^2(\Omega))$, $u_0\in \widetilde{H}^s(\Omega)$, $u_1\in \wt{L}^2(\Omega)$,  $q\in L^p(\Omega)$ with $p$ satisfying the restrictions \eqref{eq: restrictions on p} and $u\in C([0,T];\widetilde{H}^s(\Omega))\cap C^1([0,T];\wt{L}^2(\Omega))$ is a weak solution of \eqref{eq: very weak sol with potential}. Then $u$ is a very weak solution of \eqref{eq: very weak sol with potential}.
	\end{proposition}
	
	\begin{proof}
		Let us prove the result only in the case $q=0$, as the same proof applies in the general case $q\neq 0$. For the necessary modifications, we refer the reader to \cite[Proof of Proposition~4.1]{Semilinear-nonlocal-wave-paper}.

		Let $G\in L^2(0,T;\widetilde{L}^2(\Omega))$ and suppose $w\in C([0,T];\widetilde{H}^s(\Omega))\cap C^1([0,T];\wt{L}^2(\Omega))$ is the unique solution to 
		\begin{equation}
			\label{eq: PDE for w}
			\begin{cases}
				\partial_t^2w +(-\Delta)^sw =G &\text{ in }\Omega_T,\\
				w=0  &\text{ in }(\Omega_e)_T,\\
				w(T)=0, \quad \partial_t w(T)=0 &\text{ in }\Omega.
			\end{cases}
		\end{equation}
		We follow now the proof of \cite[Claim~4.2]{Semilinear-nonlocal-wave-paper}, that is we consider the parabolic regularized problems
		\begin{equation}
			\label{eq: regularization for equation of v}
			\begin{cases}
				\partial_t^2u +\eps(-\Delta)^s\partial_t u+(-\Delta)^su  = F &\text{ in }\Omega_T,\\
				u=0  &\text{ in }(\Omega_e)_T,\\
				u(0)=u_0, \quad \partial_t u(0)=u_1 &\text{ in }\Omega
			\end{cases}
		\end{equation}
		and
		\begin{equation}
			\label{eq: regularization for equation of w}
			\begin{cases}
				\partial_t^2w -\eps (-\Delta)^s \partial_t w+(-\Delta)^sw = G&\text{ in }\Omega_T,\\
				w=0  &\text{ in }(\Omega_e)_T,\\
				w(T)= \partial_t w(T)=0 &\text{ in }\Omega
			\end{cases}
		\end{equation}
		for $\eps>0$. By \cite[Theorem~3.1]{zimmermann2024calderon} or \cite[Chapter 3, Theorem~8.3]{LionsMagenesVol1}, these regularized problems have a unique (weak) solution 
		\begin{equation}
			\begin{split}
				u_{\eps}&\in C([0,T];\widetilde{H}^s(\Omega))\text{ with } \begin{cases}
					&\hspace{-0.3cm}\partial_t u_{\eps}\in L^2(0,T;\widetilde{H}^s(\Omega))\cap C([0,T];\wt{L}^2(\Omega))\\
					&\hspace{-0.3cm}\partial_t^2 u_{\eps}\in L^2(0,T;H^{-s}(\Omega))
				\end{cases} \\
				w_{\eps}&\in C([0,T];\widetilde{H}^s(\Omega))\text{ with } \begin{cases}
					&\hspace{-0.3cm}\partial_t w_{\eps}\in L^2(0,T;\widetilde{H}^s(\Omega))\cap C([0,T];\wt{L}^2(\Omega))\\
					&\hspace{-0.3cm}\partial_t^2 w_{\eps}\in L^2(0,T;H^{-s}(\Omega))
				\end{cases}
			\end{split}
		\end{equation}
		and as $\eps\to 0$, one has
		\begin{equation}
			\label{eq: convergence}
			\begin{split}
				u_{\eps}\to u &\text{ in }C([0,T];\widetilde{H}^s(\Omega)),\\
				\partial_t u_{\eps}\to \partial_t u &\text{ in }C([0,T];\wt{L}^2(\Omega)),\\
				\partial_t^2u_{\eps}\weak \partial_t^2 u&\text{ in } L^2(0,T;H^{-s}(\Omega))
			\end{split}
		\end{equation}
		(cf.~\cite[Chapter 3, eq.~(8.74)]{LionsMagenesVol1}). This ensures that
		\begin{equation}
			\label{eq: convergence second time derivative}\partial_t^2u_{\eps}\weakstar \partial_t^2 u \text{ in } L^2(0,T;H^{-s}(\Omega)) \text{ as }\eps \to 0.
		\end{equation}
		The convergence results in \eqref{eq: convergence} hold for the functions $w_\eps$ and $w$ as well. Now, using twice integration by parts, which is allowed by the regularity of the first time derivative of $u_\eps$ and $w_\eps$, we obtain
		\[
		\int_0^T \left\langle \partial_t^2 u_{\eps},w_\eps \right\rangle \,dt=\int_0^T \left\langle \partial_t^2w_\eps,u_\eps \right\rangle \,dt-\langle u_1,w_{\eps}(0)\rangle_{L^2(\Omega)}+\langle u_0,\partial_t w_{\eps}(0)\rangle_{\widetilde{H}^s(\Omega)}
		\]
		for any $\eps>0$ (cf.~\cite[eq. (4.1)]{zimmermann2024calderon}). In this computation, we used the final and initial time conditions for $w_{\eps}$ and $u_{\eps}$, respectively. Now, passing to the limit $\eps\to 0$ gives
		\[
		\int_0^T \left\langle \partial_t^2 u,w \right\rangle dt=\int_0^T \left\langle \partial_t^2w,u\right\rangle dt-\langle u_1,w(0)\rangle_{L^2(\Omega)}+\langle u_0,\partial_t w(0)\rangle_{L^2(\Omega)}.
		\]
		By \eqref{eq: very weak sol with potential} and  \eqref{eq: PDE for w} this is equivalent to 
		\[
		\int_0^T \langle F,w \rangle \,dt=\int_0^T \left\langle G,u\right\rangle dt-\langle u_1,w(0)\rangle_{L^2(\Omega)}+\langle u_0,\partial_t w(0)\rangle_{L^2(\Omega)}.
		\]
		Hence, we can conclude the proof.
	\end{proof}
	
	\begin{corollary}[Regular very weak solutions = weak solutions]
		\label{cor: regular very weak sol are weak sol}
		Suppose that $F\in L^2(0,T;\wt{L}^2(\Omega))$, $u_0\in \widetilde{H}^s(\Omega)$, $u_1\in \wt{L}^2(\Omega)$, $q\in L^p(\Omega)$ with $p$ satisfying the restrictions \eqref{eq: restrictions on p} and $u$ is a very weak solution to \eqref{eq: very weak sol with potential} such that $u\in C([0,T];\widetilde{H}^s(\Omega))\cap C^1([0,T];\wt{L}^2(\Omega))$. Then $u$ is a weak solution of \eqref{eq: very weak sol with potential}.
	\end{corollary}
	
	\begin{proof}
		By the usual well-posedness result, the problem \eqref{eq: very weak sol with potential} has a unique weak solution $v$. Using Proposition~\ref{prop: weak sol are very weak sol}, one sees that $v$ is a very weak solution to the same problem. By the uniqueness of very weak solutions, it follows that $u=v$, which in turn implies the assertion.
	\end{proof}

	\section{Runge approximation and inverse problem for linear NWEQs}\label{sec: runge+inverse}
	
	As we mentioned in Section~\ref{sec: introduction}, a key ingredient to studying nonlocal inverse problems is based on the Runge approximation. In this section, we establish the proof of Theorem~\ref{thm: runge}.
	
	\subsection{Runge approximation}
	
	\begin{proof}[Proof of Theorem~\ref{thm: runge}]
		As usual, we show the Runge approximation property by a Hahn-Banach argument. Hence, 
        we need to show that given $F\in L^2(0,T;H^{-s}(\Omega))$ vanishing on $\mathscr{R}_W$, it follows that $F=0$.
        First observe that if $u$ solves \eqref{eq:runge with zero initial}, then $v=u-\varphi$ is the unique solution to 
		\begin{equation}
			\begin{cases}
				\partial_t^2v +(-\Delta)^sv+qv = -(-\Delta)^s\varphi &\text{ in }\Omega_T,\\
				v=0  &\text{ in }(\Omega_e)_T,\\
				v(0)=\partial_t v(0)=0 &\text{ in }\Omega.
			\end{cases}
		\end{equation}
		Now, by Theorem~\ref{eq: thm well-posedness with q nonzero} there is a unique solution $w$ of
		\begin{equation}
			\label{eq: solution to weak RHS}
			\begin{cases}
				\partial_t^2w +(-\Delta)^sw +qw= F &\text{ in }\Omega_T,\\
				w=0  &\text{ in }(\Omega_e)_T,\\
				w(T)= \partial_t w(T)=0 &\text{ in }\Omega.
			\end{cases}
		\end{equation}
		As $\chi_{\Omega}(-\Delta)^s\varphi\in L^2(0,T;\wt{L}^2(\Omega))$, we can use $v$ as a test function for the equation of $w$ to obtain
		\[
		-\int_0^T \left\langle w(t),(-\Delta)^s\varphi (t)\right\rangle_{L^2(\Omega)}\,dt=\int_0^T\langle F(t),v(t)\rangle\,dt
		\]
		(see~\eqref{eq: def of very weak sol}). By assumption, the right-hand side vanishes, and hence
		\[
		\int_0^T\left\langle w(t),(-\Delta)^s\varphi (t)\right\rangle_{L^2(\Omega)}\,dt=0.
		\]
		By taking $\varphi(x,t)=\eta(t)\psi(x)$ with $\eta\in C_c^{\infty}((0,T))$ and $\psi\in C_c^{\infty}(W)$, this implies
		\[
		(-\Delta)^s w(t)=0\quad \text{ for }x\in W\text{ and a.e. }t\in [0,T].
		\]
		As $w\in L^2(0,T;\widetilde{L}^2(\Omega))$ we know that $w=0$ in $\Omega_e$ and hence the unique continuation principle ensures $w=0$ in $\R^n$. Now, according to Proposition~\ref{prop: very weak sol are distribution}, very weak solutions are distributional solutions, and thus we deduce that 
		\[
		\int_0^T\langle F,\Phi\rangle\,dt=0
		\]
		for all $\Phi\in C_c^{\infty}([0,T)\times \Omega)$ (see~\eqref{eq: distributional solution}). This in particular shows that $F=0$ as $C_c^{\infty}(\Omega_T)$ is dense in $L^2(0,T;\widetilde{H}^s(\Omega))$ and $(L^2(0,T;\widetilde{H}^s(\Omega)))^*=L^2(0,T;H^{-s}(\Omega))$.
		This proves the assertion.
	\end{proof}
	
	\subsection{DN maps for NWEQs}
	
	Let us next recall the rigorous definition of the DN map related to NWEQs.

	\begin{definition}[DN map]
		Let $\Omega\subset \R^n$ be a bounded Lipschitz domain, $s>0$ a non-integer, $T>0$ and $q\in L^p(\Omega)$ with $p$ satisfying the restrictions \eqref{eq: restrictions on p}. Then we define the DN map $\Lambda_{q}$ related to
		\begin{equation}
			\label{eq: well-posedness linear case DN}
			\begin{cases}
				\partial_t^2u +(-\Delta)^su+qu=0 &\text{ in }\Omega_T,\\
				u=\varphi&\text{ in }(\Omega_e)_T,\\
				u(0)= \partial_t u(0)=0  &\text{ in }\Omega
			\end{cases}
		\end{equation}
		by
		\begin{equation}
			\label{eq: modified linear DN map}
			\begin{split}
				\left\langle \Lambda_{q}\varphi,\psi \right\rangle \vcentcolon =	\int_{\R^n_T}(-\Delta)^{s/2}u\, (-\Delta)^{s/2}\psi\,dxdt,
			\end{split}
		\end{equation}
		for all $\varphi,\psi\in C_c^{\infty}((\Omega_e)_T)$, where $u\in C([0,T];H^s(\R^n))\cap C^1([0,T];L^2(\R^n))$ is the unique solution of \eqref{eq: well-posedness linear case DN}
		(see~\cite[Theorem~3.1~\&~Remark~3.2]{Semilinear-nonlocal-wave-paper}).
	\end{definition}
	
	\begin{remark}
		\label{remark on DN map}
		Let us mention for our later study of nonlinear NWEQs that if $qu$ is replaced by a nonlinear function $f(x,u)$ such that $f\colon \Omega\times\R\to\R$ satisfies Assumption~\ref{main assumptions on nonlinearities}, then we can still define the DN map $\Lambda_f$ by \eqref{eq: modified linear DN map}
		for all $\varphi,\psi\in C_c^{\infty}((\Omega_e)_T)$, where this time $u\in C([0,T];H^s(\R^n))\cap C^1([0,T];L^2(\R^n))$ is the unique solution of
		DN map $\Lambda_{f}$ related to \eqref{eq: main} (see~\cite[Proposition~3.7]{Semilinear-nonlocal-wave-paper}).
	\end{remark}
	
	\subsection{Proof of Theorem~\ref{thm: uniqueness potential}}
	
	Before giving the proof of Theorem~\ref{thm: uniqueness potential}, let us introduce the following notation 
	\begin{equation}
		\label{eq: time reversal}
		u^{\star}(x,t)=u(x,T-t)
	\end{equation}
	for the time reversal of the function $u\colon \R^n_T\to\R$. We first derive a suitable integral identity (cf.~e.g.~\cite[Lemma~2.4]{KLW2022} or \cite[Lemma~4.3]{zimmermann2024calderon}) and use our improved Runge approximation (Theorem~\ref{thm: runge}) to conclude the desired result.
	
	\begin{lemma}[Integral identity]
		For any $\varphi_1, \varphi_2 \in C^\infty_c((W_1)_T)$, there holds that 
		\begin{equation}
			\label{eq: integral identity}
			\left\langle(\Lambda_{q_1}-\Lambda_{q_2})\varphi_1,\varphi_2^{\star}\right\rangle=\int_{\Omega_T} \LC q_1-q_2\RC \LC u_1-\varphi_1\RC \LC u_2-\varphi_2\RC^{\star}\, dxdt,
		\end{equation}
		where $u_j$ is the unique solution of 
		\[
		\begin{cases}
			\LC \partial_t^2 +(-\Delta)^s + q_j\RC u = 0 &\text{ in }\Omega_T \\
			u=\varphi_j  &\text{ in }(\Omega_e)_T,\\
			u(0)= \partial_t u(0)=0 &\text{ in }\Omega
		\end{cases}
		\]
		for $j=1,2$. 
	\end{lemma}
	
	\begin{proof}
		
		Let us start by observing that the function $(u_2-\varphi_2)^\star$ is the unique solution of 
		\begin{equation}
			\begin{cases}
				\LC \partial_t^2 +(-\Delta)^s + q_2 \RC u = -(-\Delta)^s\varphi_2^\star &\text{ in }\Omega_T \\
				u=0  &\text{ in }(\Omega_e)_T,\\
				u(T)= \partial_t u(T)=0 &\text{ in }\Omega
			\end{cases}
		\end{equation}
		and thus by \cite[Claim~4.2]{Semilinear-nonlocal-wave-paper} we know that there holds
		\begin{equation}
			\label{eq: integration by parts}
			\int_0^T \left\langle \partial_t^2 (u_1-\varphi_1),(u_2-\varphi_2)^\star \right\rangle\,dt=\int_0^T \left\langle \partial_t^2(u_2-\varphi_2)^\star,(u_1-\varphi_1)\right\rangle dt.
		\end{equation}
		Thus, using the PDEs for $u_1-\varphi_1$ and $(u_2-\varphi_2)^\star$, \eqref{eq: integration by parts} and the symmetry of the fractional Laplacian, we may calculate
		\small
		\begin{equation}
			\begin{split}
				&\quad \, \int_{\Omega_T}(q_1-q_2)(u_1-\varphi_1)(u_2-\varphi_2)^\star\,dxdt\\
				&=
				-\int_0^T\left\langle(\partial_t^2+(-\Delta)^s)(u_1-\varphi_1),(u_2-\varphi_2)^\star\right\rangle\,dt \\
				&\quad\, +\int_0^T\left\langle(\partial_t^2+(-\Delta)^s)(u_2-\varphi_2)^\star,u_1-\varphi_1\right\rangle dt\\
				&\quad \,-\int_0^T\left\langle(-\Delta)^s\varphi_1,(u_2-\varphi_2)^\star\right\rangle\,dt+\int_0^T\left\langle(-\Delta)^s\varphi_2^\star,u_1-\varphi_1\right\rangle dt\\
				&=
				-\int_0^T\left\langle(-\Delta)^s(u_1-\varphi_1),(u_2-\varphi_2)^\star\right\rangle\,dt+\int_0^T\left\langle(-\Delta)^s(u_2-\varphi_2)^\star,u_1-\varphi_1\right\rangle dt\\
				&\quad \,-\int_0^T\left\langle(-\Delta)^s\varphi_1,(u_2-\varphi_2)^\star\right\rangle\,dt+\int_0^T\left\langle(-\Delta)^s\varphi_2^\star,u_1-\varphi_1\right\rangle dt\\
				&=-\int_0^T\langle(-\Delta)^s u_2,\varphi_1^\star\rangle\,dt+\int_0^T\langle (-\Delta)^s u_1,\varphi_2^\star\rangle\,dt.
			\end{split}
		\end{equation}
		\normalsize
		\normalsize
		By definition of the DN map, we deduce that
		\begin{equation}
			\label{eq: self adjointness DN map}
			\left\langle\Lambda_{q_1}\varphi_1,\varphi_2^\star\right\rangle-\left\langle \Lambda_{q_2}\varphi_2,\varphi_1^\star\right\rangle=\int_{\Omega_T}(q_1-q_2)(u_1-\varphi_1)(u_2-\varphi_2)^\star\,dxdt,
		\end{equation}
		which completes the proof.
	\end{proof}
	
	\begin{proof}[Proof of Theorem \ref{thm: uniqueness potential}]
		We prove Theorem \ref{thm: uniqueness potential} by using the Runge argument. Using Theorem~\ref{thm: runge}, we can approximate any $\Psi_j \in C_c^{\infty}(\Omega)$, $j=1,2$, in $L^2(0,T;\widetilde{H}^s(\Omega))$ by sequences in the Runge sets $\mathscr{R}_{W_1}$ and $\mathscr{R}_{W_2}$, respectively. Since $q_j\in L^p(\Omega)$ satisfies the estimate \eqref{eq: computation for L2 estimate},  we may pass in \eqref{eq: integral identity} to the limit and hence taking the condition \eqref{same DN map in thm linear} into account, we arrive at 
		\[
		\int_{\Omega_T}(q_1-q_2)\Psi_1\Psi_2^\star \,dxdt=0
		\]
		for all $\Psi_1,\Psi_2\in C_c^{\infty}(\Omega_T)$. This ensures that $q_1=q_2$ in $\Omega$ as we wanted to show.
	\end{proof}
	
	\section{Well-posedness and inverse problems for nonlinear NWEQs}\label{sec: nonlinear wave}
	
	In this section, we study the inverse problems for NWEQs with polyhomogeneous nonlinearities. We start in Section~\ref{subsec: unique asym exp} by showing that for any asymptotically polyhomogeneous nonlinearity the expansion is unique. Then, in Section~\ref{sec: poly with growth}, by using similar techniques as in \cite{Semilinear-nonlocal-wave-paper} and our stronger Runge approximation (Theorem~\ref{thm: runge}), we demonstrate Theorem~\ref{Thm: recovery of nonlinearity}. Let us note that in the case of asymptotically polyhomogeneous nonlinearities, Theorem~\ref{Thm: recovery of nonlinearity} only shows that the expansion coefficients coincide and not the nonlinearities themselves. This could be improved to $f^{(1)}=f^{(2)}$ in the range $2s>n$ by imposing a suitable decay of the constants $C_N$, $\sum_{k=1}^{N-1}b_k$ as $N\to\infty$, appearing in Definition~\ref{def: polyhomogeneous}, \ref{asymp}, but we do not investigate this further in this article. 
	
	\subsection{Uniqueness of asymptotic expansion}
	\label{subsec: unique asym exp}
	
	Before discussing the proof of Theorem \ref{Thm: recovery of nonlinearity}, let us make the following observation. 
	
	\begin{lemma}
		\label{lemma: unique expansion}
		Let $\Omega\subset\R^n$ be a bounded domain. Assume that we have given a sequence $\LC r_k\RC_{k=1}^\infty\subset \R$ satisfying $0<r_k<r_{k+1}$ for all $k\in\N$. Let $f\colon \Omega\times\R\to\R$ be a Carath\'eodory function and suppose that there holds
		\[
		f\sim \sum_{k\geq 1}f_k \quad \text{and}
		\quad f\sim \sum_{k\geq 1}\wt{f}_k,
		\]
		for two sequences $(f_k)_{k\in\N}$, $(\wt{f}_k)_{k\in\N}$ satisfying the assumptions in Definition~\ref{def: polyhomogeneous}, \ref{asymp} for the same sequence $(r_k)_{k\in\N}$ but with possibly different constants $C_N$ and $\widetilde{C}_N$ for $N\in\N_{\geq 2}$. Then we have $f_k=\widetilde{f}_k$ for all $k\in\N$.
	\end{lemma}
	\begin{proof}
		By assumption, we may compute
		\begin{align}
			\big|f_1(x,1) - \widetilde f_1(x,1)\big| &=
			\tau^{r_1+1}\tau^{-r_1-1} \big|f_1(x,1) - \widetilde f_1(x,1)\big|\\
			&=\tau^{-r_1-1}\big|f_1(x,\tau) - \widetilde f_1(x,\tau)\big|\\
			&\leq \tau^{-r_1-1}\big(|f(x,\tau) - \widetilde f_1(x,\tau)|
			+|f(x,\tau) - \widetilde f_1(x,\tau)|\big)\\
			&\leq \big( C_2+\widetilde C_2 \big) \tau^{r_2-r_1}
		\end{align}
		for a.e. $x\in\Omega$ and $|\tau|\leq 1$. Thus, passing to the limit $\tau\to 0$ shows that $f_1(x,1)=\widetilde f_1(x,1)$ for a.e. $x\in\Omega$ and by homogeneity $f_1=\widetilde f_1$. Continuing inductively, we obtain $f_k=\widetilde f_k$ for all $k\geq 1$. 
	\end{proof}

	\subsection{Recovery of coefficients of polyhomogeneous nonlinearities}
	\label{sec: poly with growth}

	Let us start by recalling the following continuity result on Nemytskii operators.
	
	\begin{lemma}[Continuity of Nemytskii operators, {\cite[Theorem~2.2]{ambrosetti1995primer} and \cite[Lemma~3.6]{zimmermann2024calderon}}]
		\label{lemma: Continuity of Nemytskii operators}
		Let $\Omega\subset \R^n$ be a bounded domain, $T>0$, and $1\leq q,p<\infty$. Assume that $f\colon\Omega\times \R\to\R$ is a Carath\'eodory function satisfying
		\begin{equation}
			\label{eq: cond continuity nemytskii}
			|f(x,\tau)|\leq a+b|\tau|^{\alpha}
		\end{equation}
		for some constants $a,b\geq 0$ and $0<\alpha \leq \min(p,q)$. Then the Nemytskii operator $f$, defined by
		\begin{equation}
			\label{eq: Nemytskii operator}
			f(u)(x,t)\vcentcolon = f(x,u(x,t))
		\end{equation}
		for all measurable functions $u\colon \Omega_T\to\R$, maps continuously $L^p(\Omega)$ to $L^{p/\alpha}(\Omega)$ and $L^q(0,T;L^p(\Omega))$ into $L^{q/\alpha}(0,T;L^{p/\alpha}(\Omega))$.
	\end{lemma}
	
	We will use the notation \eqref{eq: Nemytskii operator} introduced in the previous lemma.
	
	\begin{proof}[Proof of Theorem \ref{Thm: recovery of nonlinearity}]
		
		\medskip
		\noindent Our goal is to show in the first step that the equality of the DN maps for two nonlinearities $f^{(1)}$ and $f^{(2)}$ implies $f^{(1)}(x,v)=f^{(2)}(x,v)$ for any solution to a linear wave equation (see \eqref{eq:ip linear wave}). Then in a second step, we use the Runge approximation property of the solutions $v$ in $L^2(0,T; \wt H^s(\Omega))$ to deduce that all expansion coefficients $f^{(1)}_k$, $f^{(2)}_k$ for $k\in \N$ agree. In the following, we will often abbreviate $f(x,u)=\vcentcolon f(u)$.
		
		Let $\eps>0$. We start by observing that for $j=1,2$ the unique (weak) solution $u_\eps^{(j)}$ of
		\begin{equation}
			\label{eq: linearization PDE main proof}
			\begin{cases}
				\partial_t^2u +(-\Delta)^su+f^{(j)}(u)=0  &\text{in }\Omega_T,\\
				u=\eps\varphi  &\text{in }(\Omega_e)_T,\\
				u(0)= \partial_t u(0)=0 &\text{in }\Omega,
			\end{cases}
		\end{equation}
		can be expanded as $u^{(j)}_\eps = \eps v+ R^{(j)}_\eps$, where $v$ and $R^{(j)}_{\eps}$ solve
		\begin{equation}\label{eq:ip linear wave}
			\begin{cases}
				\partial_t^2v +(-\Delta)^sv=0  &\text{in }\Omega_T,\\
				v=\varphi   &\text{in }(\Omega_e)_T,\\
				v(0)= \partial_t v(0)=0 &\text{in }\Omega
			\end{cases}    
		\end{equation}
		and 
		\begin{equation}\label{eq:remainder R}
			\begin{cases}
				\partial_t^2R +(-\Delta)^s R=  -f^{(j)}\big(u^{(j)}_\eps\big)  &\text{in }\Omega_T,\\
				R=0   &\text{in }(\Omega_e)_T,\\
				R(0)= \partial_t R(0)=0 &\text{in }\Omega,
			\end{cases}    
		\end{equation}
		respectively. By \eqref{same DN map in thm}, the UCP for the fractional Laplacian and \eqref{eq: linearization PDE main proof} guarantee
		\begin{equation}
			\label{eq: same sol and remainder}
			u_{\eps}\vcentcolon = u_{\eps}^{(1)}=u_{\eps}^{(2)},\,R_{\eps}\vcentcolon = R_{\eps}^{(1)}=R_{\eps}^{(2)}\text{ and } f^{(1)}(u_\eps)=f^{(2)}(u_\eps).
		\end{equation}

        Note that we have $f^{(j)}(x,0)=0$ for a.e. $x\in\Omega$ and for $j=1,2$. For serially polyhomogeneous nonlinearities, this follows from the pointwise convergence \eqref{eq: pointwise conv ser poly} and the homogeneity of the expansion coefficients $f_k^{(j)}$ for $k\in \N$ and $j=1,2$. In contrast, for asymptotically polyhomogeneous nonlinearities it is a consequence of \eqref{eq: asymp poly} and again the homogeneity of the functions $f_k^{(j)}$. Next, we claim that we have
        \begin{equation}
        \label{eq: homogeneity estimate f j}
        |f^{(j)}(x,\tau)| \lesssim |\tau|^{r_1+1}+|\tau|^{r_{\infty}+1}
        \end{equation}
        for a.e. $x\in\Omega$ and all $\tau\in\R$. First, suppose that $f^{(j)}$ is serially polyhomogeneous. Notice that \eqref{eq: growth for bk} and the ratio test imply the convergence
        \begin{equation}
        \label{eq: conv of series bk}
            0\leq \sum_{k\geq 1}b^{(j)}_k<\infty
        \end{equation}
        and hence, using \eqref{eq: continuity nemytskii} and the Weierstrass M-test, we can conclude that
        \begin{equation}
        \label{eq: uniform conv of series}
            \sum_{k\geq 1}f^{(j)}_k\text{ converges absolutely and uniformly for a.e. }x\in\Omega\text{ and all }|\tau|\leq 1.
        \end{equation}
        On the one hand, for $|\tau|\leq 1$, we use \eqref{eq: continuity nemytskii} and \eqref{eq: conv of series bk} to derive the following estimate
        \begin{equation}
        \label{eq: homogeneity estimate serial small tau}
            |f^{(j)}(x,\tau)|\leq \sum_{k\geq 1}|f^{(j)}_k(x,\tau)| \leq  \sum_{k\geq 1}b^{(j)}_k|\tau|^{r_k+1}\leq  |\tau|^{r_1+1}\sum_{k\geq 1}b_k^{(j)}.
        \end{equation}
        On the other hand, for $|\tau|>1$, we use the pointwise convergence \eqref{eq: pointwise conv ser poly}, the homogeneity of $f_k$ and \eqref{eq: homogeneity estimate serial small tau} to get
        \begin{equation}
        \label{eq: homogeneity estimate serial large tau}
        \begin{split}
            |f^{(j)}(x,\tau)|&=\bigg|\sum_{k\geq 1}f^{(j)}_k(x,\tau)\bigg|=\bigg|\sum_{k\geq 1}|\tau|^{r_k+1}f^{(j)}_k(x,\tau/|\tau|)\bigg|\\
            &\leq  |\tau|^{r_{\infty}+1}\sum_{k\geq 1}|f^{(j)}_k(x,\tau/|\tau|)|\\
            &\leq|\tau|^{r_{\infty}+1} \sup_{\sigma=\pm 1}\sum_{k\geq 1}|f^{(j)}_k(x,\sigma)|\\
            &\leq |\tau|^{r_{\infty}+1} \sum_{k\geq 1}b_k^{(j)}.
        \end{split}
        \end{equation}
        So, combining \eqref{eq: homogeneity estimate serial small tau} and \eqref{eq: homogeneity estimate serial large tau} shows that \eqref{eq: homogeneity estimate f j} holds for serially polyhomogeneous nonlinearities up to a constant, which is given by $\sum_{k\geq 1} b_k^{(j)}$, for $j=1,2$.
        
        Next, assume that $f^{(j)}$ is asymptotically polyhomogeneous. In this case, \eqref{eq: asymp poly} demonstrates that for $|\tau|\leq 1$
        \begin{equation}
        \label{estim for f 1}
        \begin{split}
             |f^{(j)}(x,\tau)|&\leq |f^{(j)}(x,\tau)-f^{(j)}_1(x,\tau)|+|f^{(j)}_1(x,\tau)|\\
             &\leq C_2 |\tau|^{r_2+1}+ b_1|\tau|^{r_1+1} \\
             &\lesssim |\tau|^{r_1+1},
        \end{split}
        \end{equation}
        showing that the estimate \eqref{eq: homogeneity estimate f j} also holds with asymptotically polyhomogeneous case
        for a.e. $x\in\Omega$ and all $\tau\in\R$. 
        Hence, \eqref{eq: homogeneity estimate f j} also holds in this case.

		Continuing as in \cite[Proof of Theorem~1.1]{Semilinear-nonlocal-wave-paper}, we obtain the estimates
		\begin{enumerate}[(E1)]
			\item\label{item 1 uniqueness} $\left\| f^{(j)}(\cdot,u_{\eps})\right\|_{L^2(\Omega_T)} \lesssim \Vert u_\eps \Vert_{L^\infty(0,T;H^s(\R^n))}^{r_1+1} + \Vert u_\eps \Vert_{L^\infty(0,T;H^s(\R^n))}^{r_\infty+1}$,
			\item\label{item 2 uniqueness} $
			\Vert R_\eps \Vert_{L^\infty(0,T;\wt H^s(\Omega))} + \Vert \p_t R_\eps
			\Vert_{L^\infty(0,T; L^2(\Omega))}$\\
			$\lesssim \Vert u_\eps \Vert_{L^\infty(0,T;H^s(\R^n))}^{r_1+1} + \Vert u_\eps \Vert_{L^\infty(0,T;H^s(\R^n))}^{r_\infty+1}$
			\item\label{item 3 uniqueness} $ \Vert u_\eps \Vert_{L^{\infty}(0,T; H^s(\R^n))} +
			\Vert \p_t u_\eps \Vert_{L^\infty(0,T; L^2(\R^n))}\lesssim\eps$.
		\end{enumerate}
		Therefore, we deduce that 
		\begin{equation}\label{eq: estimate for R wrt eps}
			\begin{split}
				&\quad \, \Vert R_\eps \Vert_{L^\infty(0,T;\wt H^s(\Omega))} + \Vert \p_t R_\eps \Vert_{L^\infty(0,T; L^2(\Omega))}\lesssim \eps^{r_1+1}+\eps^{r_\infty+1}.    
			\end{split}
		\end{equation}
		Now, \eqref{eq: same sol and remainder} implies
		\begin{equation}
			\label{eq: identified potentials with nonlinear sol}
			f^{(1)}(\eps v + R_\eps) = f^{(2)}(\eps v+R_\eps).    
		\end{equation}
		Our next goal is to determine iteratively the terms in the series
		\[
		\sum_{k\geq 1} f^{(j)}_k(\tau),\quad j=1,2,
		\]
		starting from $f_1^{(j)}$. We distinguish two cases $2s<n$ and $2s\geq n$ as follows:\\

		\textbf{Case $2s<n$}.\\
		
		\noindent As $r_\infty$ satisfies $0<r_\infty\leq \frac{2s}{n-2s}$, we have
		\begin{equation}
			\label{eq: rinfty +1}
			1<1+r_k<1+r_\infty\leq \frac{n}{n-2s}<\frac{2n}{n-2s}=\vcentcolon p
		\end{equation}
        for all $k\in\N$.
		Using \eqref{eq: estimate for R wrt eps} and the Sobolev embedding, we see that
		\begin{equation}
			\label{eq: conv of remainder main}
			\eps^{-1}R_{\eps}\to 0\quad\text{in}\quad L^q(0,T;L^{p}(\Omega))
		\end{equation}
		as $\eps\to0$ for any $1\leq q \leq \infty$. 
		Through the estimate 
		\begin{equation}
			|f^{(j)}(x,\tau)| \leq A^{(j)} + B^{(j)}|\tau|^{r_\infty+1},
		\end{equation}
		for appropriate constants $A^{(j)},B^{(j)}\geq 0$ (see~Assumption~\ref{main assumptions on nonlinearities} or \eqref{eq: homogeneity estimate f j}), and \eqref{eq: continuity nemytskii}, we infer from Lemma~\ref{lemma: Continuity of Nemytskii operators} and \eqref{eq: rinfty +1} that the mappings
		\begin{align}
			&f^{(j)}\colon L^q(0,T; L^p(\Omega)) \to L^\frac{q}{r_\infty+1}(0,T; L^\frac{p}{r_\infty+1}(\Omega)),\\
			&f^{(j)}_k\colon L^q(0,T; L^p(\Omega)) \to L^\frac{q}{r_k+1}(0,T; L^\frac{p}{r_k+1}(\Omega)),\quad k\in\N,
		\end{align}
		are continuous as long as $q$ is chosen such that $q\geq r_\infty+1$ and $q\geq r_k+1$, respectively.

		Next, we use the inequalities $r_k< r_\ell\leq r_{\infty}$, $k\leq \ell$, and the boundedness of $\Omega$ to deduce the inclusion
		\begin{equation}
			\label{eq: embedding different rk}
			L^\frac{q}{r_k+1}(0,T; L^\frac{p}{r_k+1}(\Omega)) \hookrightarrow L^\frac{q}{r_\ell+1}(0,T; L^\frac{p}{r_\ell+1}(\Omega)),
		\end{equation}
		which ensures that the map
		\begin{equation}
			\label{eq: nemytskii argument}
			f^{(j)}_k\colon L^q(0,T; L^p(\Omega)) \to L^\frac{q}{r_\ell+1}(0,T; L^\frac{p}{r_\ell+1}(\Omega)),\quad 0<k\leq \ell\leq \infty
		\end{equation}
		is continuous.\\
		
		\noindent
		{\bf Serially polyhomogeneous nonlinearity for $\mathbf{k=1}$}:
		Multiplying \eqref{eq: identified potentials with nonlinear sol} by $\eps^{-r_1-1}$ and using the homogeneity of $f_k^{(j)}$, we have pointwise the identity
		\begin{equation}
        \label{eq: eps nonlinearity}
			\eps^{-r_1-1} f^{(j)}(u_\eps) = 
			\sum_{k=1}^\infty f_k^{(j)}(\eps^{-\frac{r_1+1}{r_k+1}} u_\eps),\quad j=1,2,    
		\end{equation}
		where 
		\begin{equation}
			\label{eq: exponents}
			\begin{cases}
				-\frac{r_1+1}{r_k+1} = -1, & \text{if }k=1,\\
				-\frac{r_1+1}{r_k+1} >-1, & \text{if }k\geq 2.
			\end{cases}
		\end{equation}
		Let us denote by $C_S>0$ the optimal Sobolev constant for the embedding $H^s(\R^n)\hookrightarrow L^p(\R^n)$ and by $D>0$ the constant in the estimate \ref{item 3 uniqueness}. Furthermore, recall that we have
		\begin{equation}
			\label{eq: assump for proof thm 16}
			\big|f^{(j)}_k(x,\tau)\big|\leq b^{(j)}_k|\tau|^{r_k+1}, \quad  r_k<r_{k+1}\leq r_{\infty}
		\end{equation}
		for $j=1,2$ (see~\eqref{eq: continuity nemytskii}). Also note that $p=2n/(n-2s)$ (see~\eqref{eq: rinfty +1}) and sufficiently large $q$ satisfy
		\[
		1\leq \frac{p}{r_\infty-r_k}<\infty,\,1\leq \frac{q}{r_\infty-r_k}<\infty
		\]
        (cf., e.g., \eqref{eq: rinfty +1}). Hence, we may decompose
		\begin{equation}
			\frac{r_\infty+1}{p}=\frac{r_k+1}{p}+\frac{r_\infty-r_k}{p}\text{ and }\frac{r_\infty+1}{q}=\frac{r_k+1}{q}+\frac{r_\infty-r_k}{q}.
		\end{equation}
		Thus, for $0<\eps\leq 1$ and $k\geq 1$, we may compute
		\begin{equation}
			\label{eq: uniform upper bound for fkueps}
			\begin{split}
				&\quad \, \big\|f_k^{(j)}\big(\eps^{-\frac{r_1+1}{r_k+1}}u_\eps\big)\big\|_{L^{\frac{q}{r_\infty+ 1}}L^{\frac{p}{r_\infty +1}}}\\
				&=\eps^{-(r_1+1)} \big\|f_k^{(j)}(u_\eps)\big\|_{L^{\frac{q}{r_\infty+ 1}}L^{\frac{p}{r_\infty +1}}}\\
				&\leq \underbrace{\eps^{-(r_1+1)}|\Omega|^{\frac{r_\infty-r_k}{p}}T^{\frac{r_\infty - r_k}{q}}\|f_k^{(j)}(u_\eps)\|_{L^{\frac{q}{r_k+1}} L^{\frac{p}{r_k+1}}}}_ {\text{by H\"older's inequality}}\\
				& \leq \underbrace{\eps^{-(r_1+1)}|\Omega|^{\frac{r_\infty-r_k}{p}}T^{\frac{r_\infty - r_k}{q}}b_k^{(j)}\|u_\eps\|_{L^q L^p}^{r_k+1}}_{\text{by }\eqref{eq: assump for proof thm 16}}\\
				& \leq \eps^{-(r_1+1)}|\Omega|^{\frac{r_\infty-r_k}{p}}T^{\frac{r_\infty +1}{q}}b_k^{(j)}\|u_\eps\|_{L^{\infty} L^p}^{r_k+1}\\
				&\leq \underbrace{\eps^{-(r_1+1)}C_s^{r_k+1}|\Omega|^{\frac{r_\infty-r_k}{p}}T^{\frac{r_\infty +1}{q}}b_k^{(j)}\|u_\eps\|_{L^{\infty} H^s}^{r_k+1}}_{\text{by Sobolev's inequality}}\\
				&\leq \underbrace{\eps^{r_k-r_1}(C_sD)^{r_k+1}|\Omega|^{\frac{r_\infty-r_k}{p}}T^{\frac{r_\infty +1}{q}}b_k^{(j)}}_{\text{by \ref{item 3 uniqueness}}}\\
				& \leq \eps^{r_k-r_1}(\max(1,C_sD))^{r_k+1}(\max(1,|\Omega|))^{\frac{r_\infty-r_k}{p}}T^{\frac{r_\infty +1}{q}}b_k^{(j)}\\
				& \leq \underbrace{(\max(1,C_sD))^{r_\infty+1}(\max(1,|\Omega|))^{\frac{r_\infty}{p}}T^{\frac{r_\infty +1}{q}}b_k^{(j)}}_{\text{by }\eps\leq 1}\\
				& =\vcentcolon M_k^{(j)},
			\end{split}
		\end{equation}
	for $k\in \N$ and $j=1,2$, where we abbreviated $L^{\alpha}(0,T;X(U))$ as $L^{\alpha}X$ for any Banach space $X(U)$ over a spatial domain $U\subset\R^n$ and $1\leq {\alpha}\leq\infty$. Since the constants in $M_k^{(j)}$ in front of $b_k^{(j)}$ are independent of $k$, we can use \eqref{eq: growth for bk} and the ratio test to get 
		\begin{equation}
			\label{eq: summability of Mk}
			\sum_{k\geq 1}M_k^{(j)}<\infty.
		\end{equation}
        Thus, the Weierstrass M-test ensures that
        \begin{equation}
        \label{eq: uniform and absolute convergence}
            \sum_{k\geq 1}f_k^{(j)}\big(\eps^{-\frac{r_1+1}{r_k+1}}u_\eps\big)\text{ converges uniformly and absolutely in }L^{\frac{q}{r_\infty+ 1}}L^{\frac{p}{r_\infty +1}},
        \end{equation}
        for $0<\eps\leq 1$.
        By Tannery's theorem\footnote{Alternatively, one can use the well-known Moore--Osgood theorem together with \eqref{eq: uniform and absolute convergence}.} \cite{loya2017amazing}, this is essentially a consequence of Lebesgue's dominated convergence theorem, \eqref{eq: eps nonlinearity}, and the fact that
        \begin{equation}
        \label{eq: convergence of nonlinearities}
        f_k^{(j)}\big(\eps^{-\frac{r_1+1}{r_k+1}}u_\eps\big)\to\begin{cases}
            0, &\text{ for }k\geq 2\\
            f^{(j)}(v),&\text{ for }k=1
        \end{cases}
        \quad \text{as }\eps\to 0
        \end{equation}
        (see \eqref{eq: uniform upper bound for fkueps} for $k\geq 2$ and \eqref{eq: conv of remainder main}, \eqref{eq: nemytskii argument} for $k=1$), we may deduce that
        \begin{equation}
			\lim_{\eps\to 0}\eps^{-r_1-1} f^{(j)}(u_\eps)=\sum_{k\geq 1}\lim_{\eps\to 0}f_k^{(j)}\big(\eps^{-\frac{r_1+1}{r_k+1}} u_\eps\big)=f_1^{(j)}(v)
		\end{equation}
		in $L^{\frac{q}{r_\infty +1}}(0,T;L^{\frac{p}{r_\infty + 1}}(\Omega))$. Finally, combining this with \eqref{eq: identified potentials with nonlinear sol} yields
		\begin{equation}
			\label{eq: equality first order terms serial}
			f_1^{(1)}(v)=f_1^{(2)}(v)
		\end{equation}
		in $L^\frac{q}{r_\infty+1}(0,T; L^\frac{p}{r_\infty+1}(\Omega))$.\\

		\noindent
		{\bf Asymptotically polyhomogeneous nonlinearity for $\mathbf{k=1}$}: 
		In this case, using Definition~\ref{def: polyhomogeneous}~\ref{asymp}, we get for $N=2$ that
        \[
		\big| f^{(j)}(x,\tau) - f_1^{(j)}(x,\tau) \big|\leq C_2 \left|\tau\right|^{r_2+1}
		\]
        when $|\tau|\leq 1$ and
        \begin{align}
        \big| f^{(j)}(x,\tau) - f_1^{(j)}(x,\tau) \big|\leq \big|f^{(j)}(x,\tau)|+|f_1^{(j)}(x,\tau)\big| \lesssim |\tau|^{r_\infty+1}
        \end{align}
        when $|\tau|>1$ for a.e. $x\in\Omega$, yielding
		\begin{equation}\label{eq: estim for f ph}
		\big| f^{(j)}(u_\eps) - f_1^{(j)}(u_\eps) \big|\lesssim \left|u_\eps\right|^{r_2+1} + \left|u_\eps\right|^{r_\infty+1}.
		\end{equation}
        Above the term $|u_\eps|^{r_2+1}$ can be estimated in $L^\frac{q}{r_\infty+1}L^\frac{q}{r_\infty+1}$ using \eqref{eq: uniform upper bound for fkueps} as
        \[
        \big\Vert |u_\eps|^{r_2+1}\big\|_{L^{\frac{q}{r_\infty+ 1}}L^{\frac{p}{r_\infty +1}}}
        \lesssim \Vert u_\eps\Vert^{r_2+1}_{L^\infty H^s}.
        \]
        The latter term of \eqref{eq: estim for f ph} can be directly estimated as
        \begin{align*}
        \Vert |u_\eps|^{r_\infty+1}\Vert_{L^{\frac{q}{r_\infty+ 1}}L^{\frac{p}{r_\infty +1}}}
        \leq \Vert u_\eps \Vert_{L^qL^p}^{r_\infty+1}
        \lesssim \Vert u_\eps \Vert_{L^\infty H^s}^{r_\infty+1}.  
        \end{align*}
        

		Multiplying the above inequality by $\eps^{-r_1-1}$ with $0<\eps<1$, recalling $r_2>r_1$, and using \ref{item 3 uniqueness} we get in the $L^\frac{q}{r_{\infty}+1}L^\frac{p}{r_{\infty}+1}$-norm
		\begin{equation}\label{eq: reco of first term}
			\begin{split}
				&\big\| \eps^{-r_1-1} f^{(j)}(u_{\eps}) - f_1^{(j)}(\eps^{-1} u_{\eps})\big\|_{L^\frac{q}{r_{\infty}+1}L^\frac{p}{r_{\infty}+1}}  \lesssim \eps^{r_2-r_1}+\eps^{r_\infty-r_1}\to 0
			\end{split}
		\end{equation}
		as $\eps\to 0$. Therefore, using \eqref{eq: nemytskii argument} and \eqref{eq: reco of first term}, we deduce that
		\begin{equation}
			f_1^{(j)}(v) =
			\lim_{\eps\to 0} f_1^{(j)}(\eps^{-1}u_\eps) = 
			\lim_{\eps\to 0} \eps^{-r_1-1} f^{(j)}(u_\eps)
		\end{equation}
		in $L^\frac{q}{r_{\infty}+1}(0,T;L^\frac{p}{r_{\infty}+1}(\Omega))$. So, taking into account \eqref{eq: identified potentials with nonlinear sol}, we get 
		\begin{equation}
			\label{eq: equality first order terms asymp}
			f_1^{(1)}(v)=f_1^{(2)}(v)
		\end{equation}
		in $L^\frac{q}{r_{\infty}+1}(0,T; L^\frac{p}{r_{\infty}+1}(\Omega))$. \\

		Next, let $\Psi\in C_c^{\infty}(\Omega_T)$. By Theorem~\ref{thm: runge}, there exists a sequence $(\psi_k)_{k\in\N}\subset C_c^{\infty}((W_1)_T)$ such that the unique solutions $(v_k)_{k\in\N}$ of 
		\[
		\begin{cases}
			\partial_t^2v_k +(-\Delta)^sv_k=0  &\text{in }\Omega_T,\\
			v_k=\psi_k  &\text{in }(\Omega_e)_T,\\
			v_k(0)= \partial_t v_k(0)=0 &\text{in }\Omega
		\end{cases}    
		\]
		satisfy $v_k-\psi_k\to \Psi$ in $L^2(0,T;\widetilde{H}^s(\Omega))$ as $k\to\infty$. Up to extracting a subsequence, we have, by Sobolev's embedding theorem, that there holds
		\[
		v_k(t)\to \Psi(t)\quad \text{in}\quad L^{p}(\Omega)
		\]
		for a.e. $t\in[0,T]$. Hence, by Lemma~\ref{lemma: Continuity of Nemytskii operators}, H\"older's inequality, and \eqref{eq: equality first order terms serial} or \eqref{eq: equality first order terms asymp}, we get
		\[
		f_1^{(1)}(\Psi(t))=\lim_{k\to\infty}f_1^{(1)}(v_k(t))=\lim_{k\to\infty}f_1^{(2)}(v_k(t))=f_1^{(2)}(\Psi(t))
		\]
		in $L^{\frac{p}{r_\infty+1}}(\Omega)$ (or $L^{\frac{p}{r_2+1}}(\Omega)$)
		for a.e. $t\in[0,T]$. As $f_1^{(j)}$ are Carath\'eodory functions this needs to hold for all $t\in [0,T]$ and hence
		\[
		f_1^{(1)}(x,\Psi(x,t))=f_1^{(2)}(x,\Psi(x,t))
		\]
		for all $(x,t)\in\Omega_T$.

		Now, let us fix $t_0\in (0,T)$ and $x_0\in\Omega$. Then we choose $\Psi(x,t)=\eta(t)\Phi(x)$ with $\eta\in C_c^{\infty}((0,T))$ and $\Phi\in C_c^{\infty}(\Omega)$, where $\eta,\Phi$ satisfy $\eta(t)=1$ in a neighborhood of $t_0$ and $\Phi(x)=1$ in a neighborhood of $x_0$. Therefore, evaluating the previous relation at $t=t_0$ we obtain
		\[
		f_1^{(1)}(x,\Phi(x))=f_1^{(2)}(x,\Phi(x))
		\]
		for a.e. $x\in\Omega$. This gives $f_1^{(1)}(x_0,1)=f_1^{(2)}(x_0,1)$. Now, the homogeneity assumptions on $f_k^{(j)}$ ensures that
		$f_1^{(1)}(x,\rho)=f_1^{(2)}(x,\rho)$ for all $x\in \Omega$ and $\rho\in\R$.
		
		Using a similar approach, one can inductively recover the higher-order terms. In fact, one can argue as follows.\\
		
		\noindent
		{\bf Serially polyhomogeneous nonlinearity for $\mathbf{k\geq 2}$}: 
		We start by introducing the function
		\[
		f^{(j),2}\vcentcolon = f^{(j)}-f^{(1)}_1= \sum_{k\geq 2}f_k^{(j)}
		\]
        for $j=1,2$ and we observe that $f^{(j),2}$ is again Carath\'eodory function satisfying the same properties as $f^{(j)}$ itself. Then, $f^{(1)}_1=f^{(2)}_1$ and \eqref{eq: same sol and remainder} imply that
		\[
		f^{(1),2}(u_\eps)=f^{(2),2}(u_\eps).
		\]
		Repeating the same proof as above, but this time multiplying with $\eps^{-r_2-1}$, we deduce $f^{(1)}_2=f^{(2)}_2$. Thus, iteratively, we get $f^{(1)}_k=f^{(2)}_k$ for any $k\in\N$.\\
		
		\noindent
		{\bf Asymptotically polyhomogeneous nonlinearity for $\mathbf{k\geq 2}$}:  Using $f^{(1)}_1=f^{(2)}_1$, \eqref{eq: asymp poly} for $N=3$, Sobolev's embedding and \ref{item 3 uniqueness}, we get
		\[
		\begin{split}
			&\quad\, \|(\eps^{-r_2-1}f^{(1)}(u_\eps)-f_2^{(1)}(\eps^{-1}u_\eps))-(\eps^{-r_2-1}f^{(2)}(u_\eps)-f_2^{(2)}(\eps^{-1}u_\eps))\|_{L^{\frac{q}{r_{\infty}+1}}L^{\frac{p}{r_{\infty}+1}}}\\
            &=\eps^{-r_2-1}\|(f^{(1)}(u_\eps)-f_2^{(1)}(u_\eps))-(f^{(2)}(u_\eps)-f_2^{(2)}(u_\eps))\|_{L^{\frac{q}{r_{\infty}+1}}L^{\frac{p}{r_{\infty}+1}}}\\
            &=\eps^{-r_2-1}\|(f^{(1)}(u_\eps)-f_2^{(1)}(u_\eps)-f_1^{(1)}(u_\eps))-(f^{(2)}(u_\eps)-f_2^{(2)}(u_\eps)-f_1^{(2)}(u_\eps))\|_{L^{\frac{q}{r_{\infty}+1}}L^{\frac{p}{r_{\infty}+1}}}\\
			&\leq \eps^{-r_2-1}\sum_{j=1,2}\|f^{(j)}(u_\eps)-f^{(j)}_2(u_\eps)-f_1^{(j)}(u_\eps)\|_{L^{\frac{q}{r_{\infty}+1}}L^{\frac{p}{r_{\infty}+1}}}\\
            &\lesssim \eps^{-r_2-1}\|u_\eps\|^{r_3+1}_{L^{q}L^{p}}\\
			&\lesssim \eps^{r_3-r_2}\to 0
		\end{split}
		\]
		as $\eps\to 0$. Taking into account \eqref{eq: nemytskii argument}, this guarantees
		\[
		\begin{split}
			f_2^{(1)}(v)-f_2^{(2)}(v)&=\lim_{\eps\to 0}(f_2^{(1)}(\eps^{-1}u_\eps)-f_2^{(2)}(\eps^{-1}u_\eps))\\
			&=\lim_{\eps\to 0}\eps^{-r_2-1}(f^{(1)}(u_\eps)-f^{(2)}(u_\eps))=0
		\end{split}
		\]
		in $L^{\frac{q}{r_{\infty}+1}}L^{\frac{p}{r_{\infty}+1}}$. Now, one can repeat the above argument to find that $f_2^{(1)}=f_2^{(2)}$. Therefore, we iteratively get $f_k^{(1)}=f_k^{(2)}$ for all $k\in\N$.\\
		
		\textbf{Case $2s\geq n$}.\\
		
		\noindent The proof is almost the same as the Sobolev embedding guarantees that we have  $H^s(\R^n)\hookrightarrow L^p(\R^n)$ for any $2\leq p<\infty$ (see \cite{Ozawa} for the critical case $2s=n$). Moreover, let us note that in the supercritical case $2s>n$, we only need \eqref{eq: asymp poly} for $|\tau|\leq 1$ by the Sobolev embedding $H^s(\R^n)\hookrightarrow L^{\infty}(\R^n)$ and the estimate \ref{item 3 uniqueness}. \\
		
		Hence, we have shown that the expansion coefficients of the nonlinearities $f^{(j)}$, $j=1,2$, coincide in both cases and we can conclude the proof.
	\end{proof}
	
	\appendix
	
	\section{Proof of Lemma~\ref{spectral lemma}}
	\label{sec: proof of spectral lemma}
	
	In this appendix, we provide the proof of the spectral theoretic lemma, Lemma~\ref{spectral lemma}. We again denote by $\langle\cdot,\cdot\rangle$ the duality pairing between $\widetilde{H}^s(\Omega)$ and $H^{-s}(\Omega)$,
    where the spaces $\widetilde{H}^s(\Omega)$, $H^{-s}(\Omega)$ are
    endowed with the norms $\|\cdot\|_{\widetilde{H}^s(\Omega)}$ and $\norm{\cdot}_{H^{-s}(\Omega)}$ as before. 
	
	\begin{proof}[Proof of Lemma~\ref{spectral lemma}]
		We start by constructing the sequence of eigenvalues $(\lambda_k)_{k\in\N}$.
		
		Let $L\colon \wt{L}^2(\Omega)\to \wt{L}^2(\Omega)$ be the compact self-adjoint operator given by $L=K\circ S$, where
		\begin{equation}
			\begin{split}
				S\colon \wt{L}^2(\Omega)\to \widetilde{H}^s(\Omega), \quad F\mapsto u
			\end{split}
		\end{equation} 
		is the source-to-solution map of the problem
		\[
		\begin{cases}
			(-\Delta)^su =F &\text{ in }\Omega,\\
			u=0  &\text{ in }\Omega_e
		\end{cases}
		\]
		and $K\colon \widetilde{H}^s(\Omega)\to \wt{L}^2(\Omega)$ denotes the usual inclusion, which is compact by the Rellich-Kondrachov theorem. Note that the solution map $S$ is well-defined and continuous by the Lax-Milgram theorem. Hence, it is clear that $L$ is compact. The operator $L$ is also self-adjoint, because the related bilinear form to $(-\Delta)^s$ is symmetric. Furthermore, if $F\in \wt{L}^2(\Omega)$ and $u=LF$, then we have
		\begin{equation}
			\label{eq: positivity}
			\langle LF,F\rangle_{L^2(\Omega)}=\langle u,F\rangle_{L^2(\Omega)}=\big\langle (-\Delta)^{s/2}u,(-\Delta)^{s/2} u \big\rangle_{L^2(\R^n)}=\|u\|_{\widetilde{H}^s(\Omega)}^2\geq 0.
		\end{equation}
		If $F\neq 0$, then we have $\langle LF,F\rangle_{L^2(\Omega)}>0$ as otherwise $u$ would vanish and hence $F=0$. Therefore, $L$ is positive definite with $\ker L=\{0\}$. By the spectral theory for compact self-adjoint operators, we deduce that $\sigma(L)=\sigma_p(L)\subset \R_+$ is at most countable with accumulation point $\mu=0$. Here, $\sigma_p(L)$ denotes the point spectrum of $L$, that is, the set of eigenvalues. Moreover, for any $\mu\in \sigma_p(L)$ its related eigenspace $\text{ker}(L-\mu)$ is finite dimensional. Next, observe that $\mu>0$ is an eigenvalue of $L$ if and only if $\lambda=1/\mu$ is an eigenvalue for $(-\Delta)^s$ and $F\in \wt{L}^2(\Omega)$ is an eigenfunction of $L$ with eigenvalue $\mu>0$ if and only if $u\vcentcolon =SF\in \widetilde{H}^s(\Omega)$ is an eigenfunction of $(-\Delta)^s$ with eigenvalue $\lambda=1/\mu$. Therefore, we may conclude that $\sigma_p((-\Delta)^s)$ is an unbounded countable sequence and the corresponding eigenspaces are finite-dimensional. \\
		
		\textit{Step 1. The first eigenvalue.}\\ 
		
		\noindent Let us define
		\begin{equation}
			\label{eq: smallest eigenvalue}
			\lambda_1=\inf \big\{\|u\|_{\widetilde{H}^s(\Omega)}^2\,;\, u\in \widetilde{H}^s(\Omega),\,\|u\|_{L^2(\Omega)}=1 \big\}.
		\end{equation}
		We assert that $\lambda_1>0$ is the smallest eigenvalue associated to $(-\Delta)^s$. To see this, let $(u_k)_{k\in\N}\subset \widetilde{H}^s(\Omega)$ be a minimizing sequence, that is
		\[
		\left\|u_k\right\|_{L^2(\Omega)}=1\text{ and }\lim_{k\to\infty}\left\|u_k \right\|_{\widetilde{H}^s(\Omega)}^2=\lambda_1.
		\]
		In particular, this implies that $\LC u_k\RC_{k\in\N}\subset\widetilde{H}^s(\Omega)$ is uniformly bounded and hence up to extracting a subsequence there exists $\phi_1\in \widetilde{H}^s(\Omega)$ such that $u_k\weak \phi_1$ in $\widetilde{H}^s(\Omega)$ as $k\to\infty$. Up to extraction of a further subsequence, we can assume by the Rellich--Kondrachov theorem that $u_k\to \phi_1$ in $\wt{L}^2(\Omega)$ as $k\to\infty$ (we still denote the subsequence by $\LC u_k\RC_{k\in \N}$). The latter condition guarantees $\left\|\phi_1\right\|_{L^2(\Omega)}=1$. Additionally, the lower semicontinuity of weak convergence ensures that $\|\phi_1\|_{\widetilde{H}^s(\Omega)}^2=\lambda_1$. Thus, $\phi_1\in \widetilde{H}^s(\Omega)$ is a minimizer of the convex functional $u\mapsto \|u\|_{\widetilde{H}^s(\Omega)}^2$, whose Euler--Lagrange equation is
		\[
		\big\langle (-\Delta)^{s/2}\phi_1,(-\Delta)^{s/2}v \big\rangle_{L^2(\R^n)}=\lambda_1\left\langle \phi_1,v \right\rangle_{L^2(\R^n)}
		\]
		for all $v\in \widetilde{H}^s(\Omega)$ (see \cite[Theorem~2.1]{KRZ-2023}). This is nothing else than the assertion that $\lambda_1>0$ is an eigenvalue and $\phi_1\in \widetilde{H}^s(\Omega)$ is a related eigenfunction. That is, $\phi_1$ solves
		\[
		\begin{cases}
			(-\Delta)^su =\lambda_1 u &\text{ in }\Omega,\\
			u=0  &\text{ in }\Omega_e.
		\end{cases}
		\]
		Next, we show that $\lambda_1>0$ is the smallest eigenvalue. For this purpose, assume that $\lambda>0$ is any eigenvalue with normalized eigenfunction $\psi\in \widetilde{H}^s(\Omega)$. Then we have 
		\[
		\|\psi\|_{L^2(\Omega)}=1\text{ and } \|\psi\|_{\widetilde{H}^s(\Omega)}^2=\lambda.
		\]
		By the definition of $\lambda_1$, we get $\lambda\geq \lambda_1$.\\
		
		\textit{Step 2. The $k$-th eigenvalue}. \\
		
		\noindent Let $k\geq 2$. Then we define
		\begin{equation}
			\label{eq: kth eigenvalue}
			\begin{split}
				\lambda_k =&\inf\Big\{\|u\|_{\widetilde{H}^s(\Omega)}^2\,;\,u\in \widetilde{H}^s(\Omega),\,\|u\|_{L^2(\Omega)}=1,\\
				&\qquad\quad\quad\quad\quad\quad\, \left\langle u,\phi_{\ell}\right\rangle_{L^2(\Omega)}=0\text{ for }1\leq \ell\leq k-1\Big\},
			\end{split}
		\end{equation}
		where $\phi_1,\ldots,\phi_{k-1}$ are the normalized eigenfunctions corresponding to the eigenvalues $\lambda_1,\ldots,\lambda_{k-1}$ and for all $1\leq \ell\leq k-1$ we have 
		\begin{equation}
			\label{eq: orthogonality relations}
			\left\langle \phi_{\ell},\phi_m\right\rangle_{L^2(\Omega)}=0\text{ for }1\leq m\leq \ell-1.
		\end{equation}
		Let us assume that that statement holds for $k-1$ and we aim to prove that it holds for $k$. As above, we take a minimizing sequence $\LC u_\ell\RC_{\ell\in\N}$ of \eqref{eq: kth eigenvalue} and, up to subtracting a subsequence, we can assume that 
		\[
		u_\ell\weak \phi_k \quad\text{in}\quad\widetilde{H}^s(\Omega)\quad\text{and}\quad u_\ell\to \phi_k\quad\text{in}\quad \wt{L}^2(\Omega)
		\]  
		as $\ell\to\infty$ for some $\phi_k\in \widetilde{H}^s(\Omega)$. Furthermore, one can easily see that there holds
		\begin{equation}
			\label{eq: orthonorm proper}
			\left\|\phi_k\right\|_{L^2(\Omega)}=1,\,  \left\langle \phi_k,\phi_\ell\right\rangle_{L^2(\Omega)}=0\text{ for }1\leq \ell\leq k-1\text{ and } \lambda_k=\left\|\phi_k\right\|_{\widetilde{H}^s(\Omega)}^2.
		\end{equation}
		Thus, $\phi_k$ is a minimizer with $\left\|\phi_k\right\|_{\widetilde{H}^s(\Omega)}^2=\lambda_k$. Next, let us define
		\[
		\begin{split}
			w_t\vcentcolon = \phi_k+tv-\sum_{\ell=1}^{k-1}\left\langle\phi_k+tv,\phi_{\ell}\right\rangle_{L^2(\Omega)}\phi_\ell=\phi_k+t\bigg(v-\sum_{\ell=1}^{k-1}\left\langle v,\phi_{\ell}\right\rangle_{L^2(\Omega)}\phi_{\ell}\bigg)
		\end{split}
		\]
		for $t\in\R$ and $v\in\widetilde{H}^s(\Omega)\setminus\{0\}$. Thus, we may estimate
		\[
		\begin{split}
			\left\|w_t\right\|_{L^2(\Omega)}&\geq \left\|\phi_k \right\|_{L^2(\Omega)}-|t|\bigg\|v-\sum_{\ell=1}^{k-1}\left\langle v,\phi_{\ell}\right\rangle_{L^2(\Omega)}\phi_{\ell}\bigg\|_{L^2(\Omega)}\\
			&\geq 1-|t|\bigg( \|v\|_{L^2(\Omega)}+\sum_{\ell=1}^{k-1}\big|\left\langle v,\phi_\ell\right\rangle_{L^2(\Omega)}\big|\|\phi_\ell\|_{L^2(\Omega)}\bigg)\\
			&\geq 1-k|t|\|v\|_{L^2(\Omega)}>0
		\end{split}
		\]
		as long as $|t|<\frac{1}{k}\|v\|_{L^2(\Omega)}$, where we used that $\|\phi_\ell\|_{L^2(\Omega)}=1$ for $1\leq \ell\leq k$. Therefore, we can define
		\[
		\widetilde{w}_t=\frac{w_t}{\left\|w_t\right\|_{L^2(\Omega)}}
		\]
		for $|t|<\frac{1}{k}\|v\|_{L^2(\Omega)}$. Using \eqref{eq: orthonorm proper} and setting $\widetilde{v}=v-\sum_{\ell=1}^{k-1}\left\langle v,\phi_{\ell}\right\rangle_{L^2(\Omega)}\phi_{\ell}$,
		we deduce that there holds
		\[
		\left.\frac{d}{dt}\right|_{t=0}\left\|w_t\right\|_{L^2(\Omega)}^2=2\langle\phi_k,\widetilde{v}\rangle_{L^2(\Omega)}=2 \left\langle \phi_k,v \right\rangle_{L^2(\Omega)},
		\]
		and
		\[
		\begin{split}
			0&=\left.\frac{d}{dt}\right|_{t=0}\left\|\widetilde{w}_t\right\|_{\widetilde{H}^s(\Omega)}^2\\
			&=\left.\frac{d}{dt}\right|_{t=0}\frac{\left\|w_t\right\|_{\widetilde{H}^s(\Omega)}^2}{\left\|w_t\right\|_{L^2(\Omega)}^2}\\
			&=\left.\frac{d}{dt}\right|_{t=0} \left\|w_t\right\|_{\widetilde{H}^s(\Omega)}^2- \lambda_k\left.\frac{d}{dt}\right|_{t=0}\left\|w_t\right\|_{L^2(\Omega)}^2\\
			&=2\big(\big\langle (-\Delta)^{s/2}\phi_k,(-\Delta)^{s/2}\widetilde{v}\big\rangle_{L^2(\R^n)}-\lambda_k\left\langle \phi_k,v\right\rangle_{L^2(\Omega)}\big)\\
			&=2\big(\big\langle (-\Delta)^{s/2}\phi_k,(-\Delta)^{s/2}v\big\rangle_{L^2(\R^n)}-\lambda_k\left\langle \phi_k,v\right\rangle_{L^2(\Omega)}\big).
		\end{split}
		\]
		In the last equality, we used that $\phi_\ell$ for $1\leq \ell\leq k-1$ are eigenfunctions of the fractional Laplacian and in $\wt{L}^2(\Omega)$ orthogonal to $\phi_k$ by \eqref{eq: orthonorm proper}. The above computation shows that $\lambda_k$ is an eigenvalue and $\phi_k$ a corresponding eigenfunction. 
		
		Next, we assert that $\lambda_k\to\infty$ as $k\to\infty$. Suppose for the sake of contradiction that $\lambda_k$ is uniformly bounded, so that also $\LC \phi_k\RC_{k\in\N}\subset\widetilde{H}^s(\Omega)$ is uniformly bounded. Thus, up to extracting a subsequence, $\LC\phi_k\RC_{k\in\N}$ converges in $\wt{L}^2(\Omega)$ and in particular is a Cauchy sequence. But then by the above construction, we have $\left\|\phi_k-\phi_m\right\|_{L^2(\Omega)}^2=\left\|\phi_k\right\|_{L^2(\Omega)}^2 +\left\|\phi_m\right\|_{L^2(\Omega)}^2+2\left\langle \phi_k,\phi_m\right\rangle_{L^2(\Omega)}=2$, for $k\neq m$ and thus $\LC\phi_k\RC_{k\in\N}$ cannot be Cauchy, a contradiction. Therefore, we necessarily have $\lambda_k\to\infty$ as $k\to\infty$.\\

		\textit{Step 3. Proof of \ref{spectral prop 1}}.\\  
		
		\noindent We already know that $(\phi_k)_{k\in\N}$ is an orthonormal system in $\wt{L}^2(\Omega)$. So, we only need to establish that the linear span of $(\phi_k)_{k\in\N}$ is dense in $\wt{L}^2(\Omega)$. As $\widetilde{H}^s(\Omega)$ is dense in $\wt{L}^2(\Omega)$ and $\LC \phi_k\RC_{k\in\N}\subset\widetilde{H}^s(\Omega)$, it is enough to show that every function in $\widetilde{H}^s(\Omega)$ can be approximated by elements in the linear span of $(\phi_k)_{k\in\N}$. So, let $v\in\widetilde{H}^s(\Omega)$ be any fixed function and define
		\begin{equation}
			\label{eq: def of vk}
			v_k=v-\sum_{\ell=1}^{k-1}\left\langle v,\phi_{\ell}\right\rangle_{L^2(\Omega)}\phi_\ell=v-\sum_{\ell=1}^{k-1}\lambda_\ell^{-1}\left\langle v,\phi_{\ell}\right\rangle_{\widetilde{H}^s(\Omega)}\phi_\ell
		\end{equation}
		for any $k\geq 2$. By orthonormality of $(\phi_k)_{k\in\N}$ in $\wt{L}^2(\Omega)$, we get $
		\left\langle v_k,\phi_\ell\right\rangle_{L^2(\Omega)}=0\text{ for any }1\leq \ell\leq k-1$.
		By formula  \eqref{eq: kth eigenvalue} this yields
		\begin{equation}
			\label{eq: Hs of vk}
			\left\|v_k\right\|_{\widetilde{H}^s(\Omega)}^2\geq \lambda_k \left\|v_k\right\|_{L^2(\Omega)}^2.
		\end{equation}
		Now, using the orthonormality of $\LC \phi_\ell \RC_{1\leq \ell \leq k-1}$, we may compute
		\begin{equation}
			\label{eq: estimate vk}
			\begin{split}
				\|v\|_{\widetilde{H}^s(\Omega)}^2&=\bigg\|v_k+\sum_{\ell=1}^{k-1}\lambda_{\ell}^{-1}\langle v,\phi_\ell\rangle_{\widetilde{H}^s(\Omega)}\phi_\ell\bigg\|^2_{\widetilde{H}^s(\Omega)}\\
				&=\left\|v_k\right\|_{\widetilde{H}^s(\Omega)}^2+2\sum_{\ell=1}^{k-1}\lambda_{\ell}^{-1}\left\langle v,\phi_\ell\right\rangle_{\widetilde{H}^s(\Omega)}\langle v_k,\phi_\ell\rangle_{\widetilde{H}^s(\Omega)}\\
				&\quad \, +\sum_{\ell=1}^{k-1}\sum_{\ell'=1}^{k-1}\lambda_{\ell}^{-1}\lambda_{\ell'}^{-1}\left\langle v,\phi_\ell\right\rangle_{\widetilde{H}^s(\Omega)}\left\langle v,\phi_{\ell'}\right\rangle_{\widetilde{H}^s(\Omega)}\left\langle \phi_{\ell},\phi_{\ell'}\right\rangle_{\widetilde{H}^s(\Omega)}\\
				&=\left\|v_k\right\|_{\widetilde{H}^s(\Omega)}^2+\sum_{\ell=1}^{k-1}\lambda_{\ell}^{-1}\big|\left\langle v,\phi_{\ell}\right\rangle_{\widetilde{H}^s(\Omega)}\big|^2\\
				&\geq \left\|v_k\right\|_{\widetilde{H}^s(\Omega)}^2.
			\end{split}
		\end{equation}
		Thus, by \eqref{eq: Hs of vk} we obtain
		\[
		\left\|v_k\right\|_{L^2(\Omega)}^2 \leq \lambda_k^{-1}\left\|v_k\right\|_{\widetilde{H}^s(\Omega)}^2\leq \lambda_k^{-1}\|v\|_{\widetilde{H}^s(\Omega)}^2.
		\]
		Passing to the limit, this implies $v_k\to 0$ in $\wt{L}^2(\Omega)$. Hence, we have
		\begin{equation}
			\label{eq: L2 limit}
			v=\sum_{\ell=1}^{\infty}\left\langle v,\phi_{\ell}\right\rangle_{L^2(\Omega)}\phi_\ell=\sum_{\ell=1}^{\infty}\big\langle v,\lambda_{\ell}^{-1/2}\phi_{\ell}\big\rangle_{\widetilde{H}^s(\Omega)}\big(\lambda_{\ell}^{-1/2} \phi_\ell \big)
		\end{equation}
		in $\wt{L}^2(\Omega)$.\\
		
		\textit{Step 4. Proof of \ref{spectral prop 2}}.\\
		
		\noindent First note that $\big( \lambda_k^{-1/2}\phi_k\big)_{k\in\N}\subset \widetilde{H}^s(\Omega)$ is orthonormal. This is a direct consequence of the above construction. It remains to show the density of the linear span of $\big( \lambda_k^{-1/2}\phi_k\big)_{k\in\N}$ in $\widetilde{H}^s(\Omega)$. Let us fix $v\in\widetilde{H}^s(\Omega)$ and suppose that the sequence $\left(v_k\right)_{k\geq 2}$ is defined as in  \eqref{eq: def of vk}. From the estimate \eqref{eq: estimate vk} we know that $\left(v_k\right)_{k\geq 2}$ is uniformly bounded in $\widetilde{H}^s(\Omega)$ and thus up to extracting a subsequence, we get $v_k\weak w$ in $\widetilde{H}^s(\Omega)$ for some $w\in\widetilde{H}^s(\Omega)$. The compact embedding $\widetilde{H}^s(\Omega)\hookrightarrow \wt{L}^2(\Omega)$ now gives $w=0$ as we already know from the previous step that $v_k\to 0$ in $\wt{L}^2(\Omega)$ as $k\to\infty$. As for any subsequence, there is a further subsequence with this property, we know that the whole sequence weakly converges in $\widetilde{H}^s(\Omega)$ to this limit $w=0$. By Mazur's lemma, there exists a sequence of convex linear combinations 
		\[
		w_\ell=\sum_{k=2}^{\ell}a_{k}^{(\ell)}v_k,\quad 0\leq a_k^{(\ell)}\leq 1,\quad \sum_{k=2}^{\ell}a_{k}^{(\ell)}=1
		\]
		such that $w_\ell\to 0$ in $\widetilde{H}^s(\Omega)$. Note that by \eqref{eq: def of vk} we have
		\[
		w_\ell=v-\sum_{k=2}^{\ell}\sum_{m=1}^{k-1}a_k^{(\ell)}\big\langle v,\lambda_m^{-1/2}\phi_m \big\rangle_{\widetilde{H}^s(\Omega)}\big(\lambda_m^{-1/2}\phi_m\big)
		\]
		and thus
		\[
		W_\ell=\sum_{k=2}^{\ell}\sum_{m=1}^{k-1}a_k^{(\ell)}\big\langle v,\lambda_m^{-1/2}\phi_m \big\rangle_{\widetilde{H}^s(\Omega)}\big(\lambda_m^{-1/2}\phi_m\big)\to v
		\]
		in $\widetilde{H}^s(\Omega)$ as $\ell\to\infty$. As the functions $W_\ell$ clearly belong to the span of $\big(\lambda_k^{-1/2}\phi_k\big)_{k\in\N}$, we may conclude the proof.\\
		
		\textit{Step 5. Proof of \ref{spectral prop 3}}.\\  
		
		\noindent Note that for any $G\in H^{-s}(\Omega)$ and $v\in \widetilde{H}^s(\Omega)$ we have by \ref{spectral prop 2} the identity
		\[
		\langle G,v\rangle=\sum_{k=1}^{\infty}\big\langle v,\lambda_k^{-1/2}\phi_k \big\rangle_{\widetilde{H}^s(\Omega)} \big\langle G,\lambda_k^{-1/2}\phi_k\big\rangle.
		\]
		Using the Cauchy-Schwarz inequality, we get
		\[
		\begin{split}
			|\langle G,v\rangle|&\leq \sum_{k=1}^{\infty}\big|\big\langle v,\lambda_k^{-1/2}\phi_k\big\rangle_{\widetilde{H}^s(\Omega)}\big|\lambda_k^{-1/2}\big|G_k\big|\\
			&\leq \bigg(\sum_{k=1}^{\infty}\big|\big\langle v,\lambda_k^{-1/2}\phi_k\big\rangle_{\widetilde{H}^s(\Omega)} \big|^2\bigg)^{1/2}\bigg(\sum_{k=1}^{\infty}\lambda_k^{-1}\left|G_k\right|^2\bigg)^{1/2}\\
			&= \|v\|_{\widetilde{H}^s(\Omega)}\bigg(\sum_{k=1}^{\infty}\lambda_k^{-1}\left|G_k\right|^2\bigg)^{1/2},
		\end{split}
		\]
		where we have again put $G_k=\langle G,\phi_k\rangle$ and used \cite[Corollary~5.10]{Brezis}. Hence, 
		\begin{equation}
			\label{eq: bound above}
			\|G\|_{H^{-s}(\Omega)}\leq \bigg(\sum_{k=1}^{\infty}\lambda_k^{-1}|G_k|^2\bigg)^{1/2}.
		\end{equation}
		Next, let $v\in \widetilde{H}^s(\Omega)$ be the unique solution to
		\begin{equation}
			\label{eq: solution elliptic problem}
			\begin{cases}
				(-\Delta)^s v=G &\text{ in }\Omega,\\
				v=0  &\text{ in }\Omega_e,
			\end{cases}
		\end{equation}
		which exists by the Lax-Milgram theorem, and satisfies
		\begin{equation}
			\label{eq: lax milgram argument}
			\|v\|_{\widetilde{H}^s(\Omega)}\leq \|G\|_{H^{-s}(\Omega)}.
		\end{equation}
		By Plancherel's theorem, the left-hand side can be written as
		\begin{equation}
			\label{eq: Plancherel}\|v\|_{\widetilde{H}^s(\Omega)}^2=\sum_{k=1}^{\infty}\big|\big\langle v,\lambda_k^{-1/2}\phi_k\big\rangle_{\widetilde{H}^s(\Omega)}\big|^2.
		\end{equation}
		As $v$ solves \eqref{eq: solution elliptic problem}, we get
		\begin{equation}
			\label{eq: coincidence of coeff}
			\begin{split}
				\big\langle v,\lambda_k^{-1/2}\phi_k\big\rangle_{\widetilde{H}^s(\Omega)}&=\lambda_k^{-1/2}\big\langle (-\Delta)^{s/2}v,(-\Delta)^{s/2}\phi_k\big\rangle_{L^2(\R^n)}\\
				&=\lambda_k^{-1/2}\left\langle G,\phi_k\right\rangle\\
				&=\lambda_k^{-1/2}G_k. 
			\end{split}
		\end{equation}
		Taking into account \eqref{eq: lax milgram argument} and \eqref{eq: Plancherel}, we get
		\begin{equation}
			\label{eq: bound below}
			\begin{split}
				\bigg(\sum_{k=1}^{\infty}\lambda_k^{-1}\left|G_k\right|^2\bigg)^{1/2}&=
				\bigg(\sum_{k=1}^{\infty}\big|\big\langle v,\lambda_k^{-1/2}\phi_k\big\rangle_{\widetilde{H}^s(\Omega)}\big|^2\bigg)^{1/2}\\
				&=\|v\|_{\widetilde{H}^s(\Omega)}\\
				&\leq \|G\|_{H^{-s}(\Omega)}.
			\end{split}
		\end{equation}
		Thus, we may conclude that for any $G\in H^{-s}(\Omega)$, we have
		\begin{equation}
			\label{eq: characterization of dual norm}
			\|G\|_{H^{-s}(\Omega)}=\bigg(\sum_{k=1}^{\infty}\lambda_k^{-1}\left|G_k\right|^2\bigg)^{1/2}.
		\end{equation}
		Next, we assert that for any $G\in H^{-s}(\Omega)$ we have
		\begin{equation}
			\label{eq: convergence in dual space}
			G=\sum_{k=1}^{\infty}G_k\phi_k\quad\text{in}\quad H^{-s}(\Omega),
		\end{equation}
		where $G_k=\left\langle G,\phi_k \right\rangle$, for $k\in\N$. Again, let $v\in\widetilde{H}^s(\Omega)$ be the unique solution of \eqref{eq: solution elliptic problem}. Then from \eqref{eq: bound above} and \eqref{eq: bound below}, we know that
		\begin{equation}
			\label{eq: isomorphism identity}
			\|v\|_{\widetilde{H}^s(\Omega)}=\|G\|_{H^{-s}(\Omega)}.
		\end{equation}
		Therefore the source-to-solution map $S\colon H^{-s}(\Omega)\to \widetilde{H}^s(\Omega)$ related to \eqref{eq: solution elliptic problem} is an isometric isomorphism. In fact, surjectivity follows by using $G=(-\Delta)^s v\in H^{-s}(\R^n)\hookrightarrow H^{-s}(\Omega)$ for given $v\in\widetilde{H}^s(\Omega)$ as a source. By \ref{spectral prop 2}, we already know that
		\[
		v=\sum_{k=1}^{\infty}\big\langle v,\lambda_k^{-1/2}\phi_k\big\rangle_{\widetilde{H}^s(\Omega)}\lambda_k^{-1/2}\phi_k
		\]
		in $\widetilde{H}^s(\Omega)$ for any $v\in\widetilde{H}^s(\Omega)$. As $SG=v$ and $S^{-1}$ is a bounded linear map by the Banach isomorphism theorem, we deduce that
		\[
		G=S^{-1}v=\sum_{k=1}^{\infty}\lambda_k^{-1/2}\big\langle v,\lambda_k^{-1/2}\phi_k\big\rangle_{\widetilde{H}^s(\Omega)}S^{-1}\phi_k\quad\text{in}\quad H^{-s}(\Omega).
		\]
		As $S^{-1}\phi_k=\lambda_k\phi_k$, we get by \eqref{eq: coincidence of coeff} the identity
		\begin{equation}
			\label{eq: identity dual function}
			\begin{split}
				G&=\sum_{k=1}^{\infty}\lambda_k^{1/2}\big\langle v,\lambda_k^{-1/2}\phi_k\big\rangle_{\widetilde{H}^s(\Omega)}\phi_k\\
				&=\sum_{k=1}^{\infty}G_k\phi_k\quad\text{in}\quad H^{-s}(\Omega).
			\end{split}
		\end{equation}
		This verifies the assertion \eqref{eq: convergence in dual space}. Observe that the bilinear form
		\begin{equation}
			\label{eq: inner product on dual space}
			\langle G,H\rangle_{H^{-s}(\Omega)}\vcentcolon =\langle SG,SH\rangle_{\widetilde{H}^s(\Omega)}
		\end{equation}
		for $G,H\in H^{-s}(\Omega)$ defines an inner product on $H^{-s}(\Omega)$ and the induced norm coincides with the dual norm $\|\cdot\|_{H^{-s}(\Omega)}$ (see \eqref{eq: isomorphism identity}). Note that 
		\[
		\begin{split}
			\left\langle \phi_k,\phi_\ell\right\rangle_{H^{-s}(\Omega)}&=\left\langle S\phi_k,S\phi_{\ell}\right\rangle_{\widetilde{H}^s(\Omega)}=\lambda_k^{-1}\lambda_{\ell}^{-1}\left\langle \phi_k,\phi_{\ell}\right\rangle_{\widetilde{H}^s(\Omega)}\\
			&=\lambda_{\ell}^{-1}\left\langle \phi_k,\phi_{\ell}\right\rangle_{L^2(\Omega)}=\lambda_k^{-1}\delta_{k,\ell}
		\end{split}
		\]
		for any $k,\ell\in\N$. Hence, $\big( \lambda_k^{1/2}\phi_k\big)_{k\in\N}$ is orthonormal in $H^{-s}(\Omega)$. By the definition of the isomorphism $S$, $S\phi_k=\lambda_k^{-1}\phi_k$ and \eqref{eq: inner product on dual space}, we get
		\begin{equation}
			\label{eq: expression for Gk}
			\begin{split}
				G_k&=\langle G,\phi_k\rangle\\
				&=\left\langle SG,\phi_k\right\rangle_{\widetilde{H}^s(\Omega)}\\
				&=\left\langle SG,S (\lambda_k\phi_k)\right\rangle_{\widetilde{H}^s(\Omega)}\\
				&=\lambda_k\left\langle G,\phi_k\right\rangle_{H^{-s}(\Omega)}.
			\end{split}
		\end{equation}
		Finally, by \eqref{eq: identity dual function} this implies
		\begin{equation}
			G=\sum_{k=1}^{\infty} \big\langle G,\lambda_k^{1/2}\phi_k\big\rangle_{H^{-s}(\Omega)}\lambda_k^{1/2}\phi_k\quad \text{in}\quad H^{-s}(\Omega),
		\end{equation}
		which in turn implies that $\big(\lambda_k^{1/2}\phi_k\big)_{k\in\N}$ is an orthonormal basis in $H^{-s}(\Omega)$.
	\end{proof}
	
	\medskip 
	
	\noindent\textbf{Acknowledgments.} 
	\begin{itemize}
		\item Y.-H.~Lin was partially supported by the National Science and Technology Council (NSTC), Taiwan, under project 113-2628-M-A49-003. Y.-H. Lin is also a Humboldt research fellow. 
		\item T.~Tyni was supported by the Research Council of Finland (Flagship of Advanced Mathematics for Sensing, Imaging and Modelling grant 359186) and by the Emil Aaltonen Foundation. 
		\item P.~Zimmermann was supported by the Swiss National Science Foundation (SNSF), under grant number 214500.
	\end{itemize}

    \section*{Statements and Declarations}
	
	\subsection*{Data availability statement}
	No datasets were generated or analyzed during the current study.
	
	\subsection*{Conflict of Interests} Hereby, we declare that there is no conflict of interest.

	\bibliography{refs} 
	
	\bibliographystyle{alpha}
	
\end{document}